\documentclass[12pt,reqno]{amsart}
\usepackage{amssymb,amsmath,amsthm,enumerate,verbatim,bbm}
\usepackage[a4paper]{geometry}

\DeclareMathOperator{\sign}{sign}
\DeclareMathOperator{\Ran}{Ran}

\DeclareMathOperator{\spec}{spec}

\DeclareMathOperator{\supp}{supp}

\DeclareMathOperator{\VMO}{VMO}
\DeclareMathOperator{\sing}{sing}

\renewcommand\Im{\hbox{{\rm Im}}\,}
\renewcommand\Re{\hbox{{\rm Re}}\,}
\newcommand{\abs}[1]{\lvert#1\rvert}

\newcommand{\norm}[1]{\lVert#1\rVert}

\newcommand{\jap}[1]{\langle#1\rangle}

\newcommand{\bbT}{{\mathbb T}}
\newcommand{\bbR}{{\mathbb R}}
\newcommand{\bbC}{{\mathbb C}}

\newcommand{\bbZ}{{\mathbb Z}}

\newcommand{\wh}{\widehat}

\newcommand{\ess}{\text{ess}}

\newcommand{\bb}{\mathbf{b}}
\newcommand{\bq}{\mathbf{q}}
\newcommand{\bh}{{\mathbf{h}}}
\newcommand{\bff}{{\mathbf{f}}}
\newcommand{\bg}{{\mathbf{g}}}

\newcommand{\bu}{{\mathbf{u}}}
\newcommand{\sh}{{\sf {h}}}

\newcommand{\bP}{\mathbf{P}}

\newcommand{\bW}{\mathbf{W}}

\newcommand{\bomega}{\pmb{\omega}}

\newcommand{\calB}{\mathcal{B}}
\newcommand{\calU}{\mathcal{U}}

\newcommand{\Sch}{\mathbf{S}}

\numberwithin{equation}{section}

\theoremstyle{plain}
\newtheorem{theorem}{\bf Theorem}[section]
\newtheorem*{theorem*}{Theorem 1.1$'$}
\newtheorem{lemma}[theorem]{\bf Lemma}

\newtheorem{assumption}[theorem]{\bf Assumption}
\newtheorem{corollary}[theorem]{\bf Corollary}

\theoremstyle{definition}

\theoremstyle{remark}
\newtheorem*{remark*}{\bf Remark}

\newcommand{\wt}{\widetilde}
\newcommand{\eps}{\varepsilon}
\newcommand{\loc}{\mathrm{loc}}

\newcommand{\1}{\mathbbm{1}}

\newcommand{\Hank}{ H}
\newcommand{\bHank}{{ \mathbf{H}}}  
\newcommand{\Gank}{{\Gamma}}
\newcommand{\bGank}{{\mathbf{\Gamma}}}

\begin{document}

\title[Spectral asymptotics for Hankel operators]{Spectral asymptotics for compact self-adjoint Hankel operators}

\author{Alexander Pushnitski}
\address{Department of Mathematics, King's College London, Strand, London, WC2R~2LS, U.K.}
\email{alexander.pushnitski@kcl.ac.uk}

\author{Dmitri Yafaev}
\address{Department of Mathematics, University of Rennes-1,
Campus Beaulieu, 35042, Rennes, France}
\email{yafaev@univ-rennes1.fr}

\begin{abstract} 
We describe large classes of compact self-adjoint Hankel operators whose eigenvalues have power asymptotics and obtain   explicit expressions for the coefficient in front of the leading term. The results are stated both in the discrete and continuous representations for Hankel operators. We also elucidate two key principles underpinning the proof of such asymptotic relations. We call them \emph{the localization principle} and \emph{the symmetry principle}. The localization principle says that disjoint components of the singular support of the symbol of a Hankel operator make independent contributions into the asymptotics of eigenvalues. The symmetry principle says that if   the singular support of a symbol   does not contain the points $1$ and $-1$ in the discrete case (or the points $0$ and $\infty$ in the continuous case), then the  spectrum of the corresponding  Hankel operator is asymptotically symmetric with respect to the reflection around zero.  
\end{abstract}



 \dedicatory{To the memory of Yura Safarov} 

\maketitle

\section{Introduction}\label{sec.z}


\subsection{Localization and symmetry principles}
Hankel  operators   admit various unitary equivalent descriptions. 
First we recall the definition of Hankel operators on the Hardy class $H^2(\bbT)$.
Let $\bbT$ be the unit circle in the complex plane, equipped with the 
normalized Lebesgue measure $dm(\mu)= (2\pi i \mu)^{-1} d\mu$, $\mu\in\bbT$. 
The Hardy class $H^2(\bbT)\subset L^2(\bbT)$ is defined in the standard way as
the subspace spanned by the functions $1, \mu, \mu^{2 },\dots$ in $L^2(\bbT)$. 
Let $P_+: L^2(\bbT)\to H^2(\bbT)$ be the orthogonal projection onto $H^2(\bbT)$, 
and let $W$   be  the involution in $L^2(\bbT)$ defined by $(Wf)(\mu)=f(\overline{\mu})$. 
For a function  $\omega \in L^\infty(\bbT)$, the \emph{Hankel operator} 
$H (\omega)$ with the symbol $\omega$  is defined on the Hardy class $H^2(\bbT)$ by the relation
\begin{equation}
H(\omega)f=P_+(\omega Wf), 
\quad 
f\in H^2(\bbT).
\label{a1}
\end{equation}
Background information on the theory of Hankel  operators can be found in  
the books \cite{NK,Peller}.

In this paper, we are interested in self-adjoint Hankel operators. 
Thus, we will always assume that the symbol $\omega$ satisfies the symmetry condition
\begin{equation}
\overline{\omega(\mu)}=\omega(\overline{\mu}), \quad \mu\in \bbT. 
\label{a2}
\end{equation}

It is well known that $H(\omega)$ is bounded if $\omega\in L^\infty(\bbT)$ and that $H(\omega)$ is
compact if $\omega$ is continuous.
Moreover, if $\omega \in C^\infty({\bbT})$, then the  eigenvalues  of $\Hank(\omega)$
go to zero faster than any power of $n^{-1}$ as $n\to\infty$. 
Conversely, the singularities of $\omega(\mu)$ are responsible for the power-like 
decay of eigenvalues or even for the appearance of the continuous spectrum.

The first result in this direction is due to S.~R.~Power \cite{Power}  
who considered the essential spectrum of $\Hank(\omega)$ for piecewise continuous symbols $\omega$. 
The structure of the absolutely continuous spectrum was later described by J.~S.~Howland \cite{Howland1}.
Although the assumptions of \cite{Power} and \cite{Howland1} are slightly different, in both cases
$\omega$ has jump discontinuities on the unit circle; the essential (resp. absolutely continuous) spectrum can be described 
in terms of these jumps.   
It turns out that the contributions of different jumps of $\omega$ (i.e. the jumps located at different points 
of the unit circle) to the essential spectrum are independent of each other. 
We call this fact  \emph{the localization principle}.
Further, observe that under the symmetry condition 
\eqref{a2}, the singularities of $\omega$ can be located (i) at the points $+1$ and $-1$ and (ii)
at pairs $(\zeta,\overline{\zeta})$ of complex conjugate points on the unit circle. 
It turns out that the contributions of the jumps to the essential spectrum in cases (i) and (ii) are qualitatively different. 
More precisely, the jumps  of $\omega(\mu)$ at the points $\mu=\pm 1$ yield the intervals $[0,\kappa_{\pm}]$ 
(with $\kappa_{\pm}$ determined by the size of the jumps) of the essential spectrum, while 
the jumps at each pair $(\zeta,\overline{\zeta})$ of complex conjugate points yield symmetric intervals $[-\kappa,\kappa]$. 
The last assertion is natural to call \emph{the symmetry principle}.

Our goal is to  find the asymptotic formulas for eigenvalues for wide classes of \emph{compact self-adjoint} Hankel operators. To that end, we state both  the localization principle and  the symmetry principle in a rather general setting adapted to the study of the discrete spectrum. In this paper we are interested in Hankel operators with the power-like asymptotics of eigenvalues. For Hankel operators $\Hank(\omega)$ realized in the space $H^2 (\bbT)$ by formula \eqref{a1}, such behavior occurs for symbols $\omega (\mu)$ with logarithmic singularities. 
This means that $\omega (\mu)$  is continuous at singular points $\zeta$  but the rate of convergence
$\omega (\mu)- \omega (\zeta)\to 0$  
as $\mu\to \zeta$ is logarithmic; thus, $\omega$ does not satisfy the H\"older condition with any positive exponent.

Before going into details, we describe another representation for Hankel operators as ``infinite matrices".

\subsection{Matrix Hankel operators}

Let $\omega\in L^\infty(\bbT)$, and let the operator $\Hank(\omega)$  be defined by formula \eqref{a1}. 
The ``matrix elements'' of $\Hank(\omega)$ in the orthonormal basis 
$\{\mu^j\}_{j=0}^\infty$ in $H^2(\bbT)$ are
\begin{equation}
(H(\omega) \mu^j, \mu^k)_{L^2(\bbT)}
=
\wh\omega(j+k), \quad j,k\geq0,
\label{a9}
\end{equation}
where $\wh\omega$ are the Fourier coefficients of $\omega$, 
$$
\wh \omega(j)= \int_{\bbT}  \omega(\mu) \mu^{-j} dm(\mu),
\quad 
j\in\bbZ.
$$
This gives the standard ``matrix representation'' 
for Hankel operators in the space $ \ell^2(\bbZ_+)$.

It will be convenient to introduce some notation related to this representation. 
For a   sequence $\{h(j)\}_{j=0}^\infty$     of complex numbers,
the Hankel operator $\Gank(h)$ in the space $ \ell^2(\bbZ_+)$ is formally defined by the 
``infinite matrix'' $\{h(j+k)\}_{j,k=0}^\infty$, that is,
\begin{equation}
(\Gank (h) u) (j)= \sum_{k=0}^\infty h(j+k) u (k), \quad 
u=\{u(j)\}_{j=0}^\infty\in\ell^2(\bbZ_+).
\label{a11}
\end{equation}
Of course, $\Gank(h)$ is symmetric    if the sequence $h$ is real-valued. 
By \eqref{a9}, a Hankel operator $\Gank(h)$ is unitarily equivalent to $\Hank(\omega)$ 
if and only if
\begin{equation}
\wh\omega(j)=h(j), \quad j\geq0;
\label{a12}
\end{equation}
in this case $\omega$ is called a symbol of $\Gamma(h)$. 
Since \eqref{a12} involves only $j\geq0$, a symbol is not uniquely defined. 
By Nehari's theorem \cite{Nehari}, the operator $\Gank(h)$ is bounded if and only if relation \eqref{a12} is satisfied for some some $\omega\in L^\infty (\bbT)$.

Suppose that, for some $\alpha\geq0$ and  $j\to\infty$, 
\begin{equation}
h(j)=
\biggl({\sf b}_1+{\sf b}_{-1}(-1)^{j}+2  
 \sum_{\ell=1}^L b_\ell \cos(\varphi_{\ell} j - \psi_\ell)
 \biggr)  j^{-1}(\log j)^{-\alpha}+\text{error term},  
\label{eq:AS}
\end{equation}
where $\varphi_{1},\ldots, \varphi_{L}\in (0,\pi)$ are {\it distinct} numbers and $\psi_{1},\ldots, \psi_{L}$ as well as  ${\sf b}_{-1}, {\sf b}_{1}, b_{1}, \ldots, b_{L}$ are arbitrary real numbers. It can be shown that the singular support of the symbol $\omega(\mu)$ of the corresponding 
Hankel operator  $\Gamma (h)$ (without the error term in \eqref{eq:AS}) consists of the points $\pm 1$ (if ${\sf b}_{\pm 1}\neq 0$) and the pairs $(e^{i\varphi_{\ell}},
e^{- i\varphi_{\ell}})$, $\ell=1,\ldots, L$. If $\alpha=0$, then  $\omega(\mu)$  has jumps at these points. This implies 
(see \cite{Howland1} and \cite{PYH}) that the absolutely continuous spectrum of $\Gamma (h)$ equals  
\begin{equation}
\spec_{\text{ac}} (\Gamma (h))= [ 0, \pi {\sf b}_{1} ] \cup  [ 0, \pi {\sf b}_{-1}  ] \cup \bigcup_{ \ell =1}^L  
[ -\pi b_{-\ell}, \pi b_{\ell}   ]
\label{eq:XX4}\end{equation}
(each of the intervals in the right-hand side yields  the absolutely continuous spectrum of multiplicity one). 

In the case $\alpha>0$, the singularities of $\omega(\mu)$  are weaker so that the operators $\Gamma (h)$ are compact, and the decay of their eigenvalues is determined by these singularities.
Let us describe a typical result of this paper. For a compact self-adjoint  operator $\Gamma$, let us
 denote by $\{\lambda_n^+(\Gamma)\}_{n=1}^\infty$ 
the non-increasing sequence of positive eigenvalues of $\Gamma$
and set $\lambda_n^-(\Gamma)=\lambda_n^+(- \Gamma)$. We show that
\begin{equation}
\lambda_n^\pm(\Gamma (h))
=
 a^\pm n^{-\alpha}+o(n^{-\alpha}), \quad n\to\infty,\quad \alpha>0,
\label{a3}
\end{equation}
where
\begin{equation}
a^\pm
=
\varkappa(\alpha)\bigl(({\sf b}_{-1})_\pm^{1/\alpha}+({\sf b}_{1})_\pm^{1/\alpha} 
+\sum_{\ell=1}^L \abs{b_\ell}^{1/\alpha}\bigr)^\alpha,
\label{a19}
\end{equation}
${\sf b}_\pm=(\abs{{\sf b}}\pm {\sf b})/2$,  
and the numerical coefficient $\varkappa(\alpha)$ can be expressed in terms of the Beta function, 
\begin{equation}
\varkappa(\alpha)= 2^{-\alpha}
\pi^{1-2\alpha}
\bigl(B(\tfrac{1}{2\alpha},\tfrac12)\bigr)^\alpha.
\label{a15}
\end{equation}
Thus both  in the theory of  the continuous spectrum (formula \eqref{eq:XX4}) 
and in the theory of the discrete spectrum (formula \eqref{a19}) the contributions of different terms in \eqref{eq:AS} are independent of each other. So it is natural to use the term \emph{the localization principle} for this phenomenon.

Observe that the first two terms in the right-hand side of \eqref{eq:AS} contribute to the leading term of the asymptotics  of
$\lambda_n^+(\Gamma (h))$ (resp. of $\lambda_n^-(\Gamma (h))$) only if the coefficients ${\sf b}_{1}$ or ${\sf b}_{-1}$ are positive  (resp. negative).   
On the other hand, in both  continuous   and  discrete cases, the oscillating terms 
in \eqref{eq:AS} yield symmetric contributions to the continuous spectrum of $\Gamma (h)$ and to the leading term of the asymptotics  of the eigenvalues $\lambda_n^\pm(\Gamma (h))$. 
It is natural to call this phenomenon  \emph{the symmetry principle}.

\subsection{Related work}

This paper can be considered as a continuation of our previous work \cite{II} where Hankel operators $\Gamma (h)$ with matrix elements  \eqref{eq:AS} were studied for $\alpha>0$ under the assumption that 
$b_{1}=\cdots = b_{L}=0$; in this case $h(j)$ does not contain the oscillating terms. 
Extending this to the case of non-zero coefficients 
$b_1,\dots, b_L$ turns out to be a non-trivial problem. 
In fact, we had to isolate and formalise both the localization and the symmetry principles
for compact self-adjoint Hankel operators precisely in order to handle this situation.

We note also our paper \cite{III} where  Hankel operators $\Gamma(h)$ corresponding to the
sequences  
\begin{equation}
h(j)=\biggl(\sum_{\ell=1}^L b_\ell \zeta_\ell^{-j} \biggr) j^{-1}(\log j)^{-\alpha}
+\text{error term}, \quad j\to\infty, \quad \alpha>0,
\label{eq:sv1}
\end{equation}
were considered. Here $\zeta_1,\dots,\zeta_L\in\bbT$ are distinct points
(not necessarily complex conjugate pairs of points)
and $b_1,\dots,b_L$ are any complex coefficients so that the operators $\Gamma(h)$ need not be self-adjoint. 
In \cite{III}, we studied singular values $s_n(\Gamma (h))$ of the operator $\Gamma (h)$ and obtained the asymptotic formula
\begin{equation}
s_n(\Gamma (h))= a \, n^{-\alpha}+o(n^{-\alpha}), \quad n\to\infty,
\label{a3a}
\end{equation}
with the coefficient
$$
a =
 \varkappa(\alpha)\biggl( \sum_{\ell=1}^L \abs{b_\ell}^{1/\alpha}\biggr)^\alpha.
$$
The proof of this assertion in \cite{III} used   our result of \cite{II} and also required 
the localization principle for singular values of Hankel operators. 
This principle allowed us to separate the contributions of different
terms in the right-hand side of \eqref{eq:sv1}. 

 We finally mention the fundamental paper
 \cite{MPT} where the spectra of all bounded self-adjoint Hankel operators were characterized 
 in terms of a certain balance of their positive and negative parts. In particular, the  spectra of compact  Hankel operators were characterized by the two conditions:
(i) the multiplicities of the eigenvalues $\lambda$ and $-\lambda$ do not differ by more than one;  
(ii) if the point $\lambda=0$ is an eigenvalue, then necessarily it has infinite multiplicity. 
The first of these conditions is similar in spirit to the asymptotic formulas \eqref{a3}, \eqref{a19} 
for the eigenvalues, but of course neither of these two results implies the  other one.

\subsection{Main ideas of the approach}

 The results of \cite{II, III} are the basis for our proof of the relations  \eqref{a3},  \eqref{a19}. 
 The other two ingredients are
 the localization and the symmetry principles for eigenvalues. 
 
\emph{The localization principle}  allows us to separate the contributions of different
terms in  the right-hand side of \eqref{eq:AS}. In a more general setting it
 says   that the contributions of disjoint components of 
the singular support of $\omega$ (denoted $\sing\supp\omega$) to the asymptotics of the eigenvalues of  
$\Hank(\omega)$ are independent of each other.
The precise statement is Theorem~\ref{thm.a3} below.
This principle allows one to split the singular support of $\omega$ into 
disjoint pieces and to study the operators corresponding to each piece separately. 
In our previous work \cite{III} we discussed localization principle for singular values of
(not necessarily self-adjoint) Hankel operators. The localization principle for eigenvalues
requires some new operator theoretic input, which is stated here as Theorem~\ref{thm.b2}.

\emph{The symmetry principle} is needed to treat Hankel operators $\Gamma(h)$ with the oscillating matrix elements 
$h (j)= \cos(\varphi  j - \psi)j^{-1}(\log j)^{-\alpha}$. 
 It shows that asymptotically   the sequence $s_n(\Gamma(h))$ of singular values is ``shared equally'' between the sequences  $\lambda^+_n(\Gamma(h))$ and $\lambda^-_n(\Gamma(h))$ of positive and negative eigenvalues.   So given the asymptotic  formula for $s_{n} (\Gamma(h))$  obtained in \cite{III}, we get a formula for $\lambda^\pm_{n} (\Gamma(h))$.
    More generally, the symmetry principle
 says  that if $\sing\supp\omega$ does not contain $1$ and $-1$, then the spectrum of the self-adjoint Hankel operator $H(\omega)$ is asymptotically symmetric with respect to the reflection $\lambda\mapsto-\lambda$.  
 For compact operators $\Gamma(h)$ with such symbols, this means that the leading terms of the asymptotics of   $\lambda^+(\Gamma(h))$ and $\lambda^-(\Gamma(h))$ coincide.
 
 For our purposes, it suffices to consider the case when $\sing\supp\omega$ consists of a finite number of points.
 
\subsection{The structure of the paper}

Along with the  representation in the space  $\ell^2(\bbZ_+)$, 
Hankel operators  can be defined as  integral operators ${\pmb \Gamma} (\bh)$ in $L^2(\bbR_+)$
with  kernels $\bh(t+s)$. 
We will refer to the Hankel operators $\Gamma(h)$ acting in $\ell^2(\bbZ_+)$
as to the discrete representation, and to the Hankel operators  ${\pmb \Gamma} (\bh)$ acting in $L^2(\bbR_+)$
as to the continuous representation. Similarly to the realization of operators $\Gamma(h)$ in the Hardy space $H^2 (\bbT)$ described in Section~1.1, ``continuous" Hankel operators  ${\pmb \Gamma} (\bh)$  can be realized in the Hardy space $H^2 (\bbR)$ of functions analytic in the upper half-plane. We   use boldface font for objects associated with the continuous representation.  We have tried to make exposition in the discrete and continuous cases parallel as much as possible. 

We collect necessary operator theoretic background in Section~2;
the key results of that section are Theorems~\ref{thm.b3} and \ref{lma.c2}.
In Section~\ref{sec.b} we prove the localization principle, see 
Theorems~\ref{thm.a3} and \ref{thm.b6} (one of these theorems refers to the discrete
representation and another one to the continuous one). 
In Section~\ref{sec.c} we prove the symmetry principle, see  Theorems~\ref{thm.a1}
and \ref{thm.c3} (again, those are the discrete and the continuous versions). 

Applications of these general results to Hankel operators
$\Gamma (h)$ and ${\pmb \Gamma} (\bh)$      are given in Sections~5 and 6, respectively. The main result for the operators
$\Gamma (h)$ is stated as Theorem~\ref{thm.a5}. 
Here we prove formulas \eqref{a3}, \eqref{a19} for sequences $h(j)$ with  asymptotics \eqref{eq:AS} as $j\to \infty$.
Moreover, in Section~5.4  we discuss  the version  of Theorem~\ref{thm.a5} for 
Hankel operators  acting on the Hardy space $H^2(\bbT)$.

 The results on the asymptotic behavior of eigenvalues of integral Hankel  operators
 ${\pmb \Gamma} (\bh)$  are stated similarly to the discrete case. 
 However, we have to take into account the fact that $\bh(t)$ may be singular 
 both as $t\to \infty$ and as $t\to 0$. 
 Note that oscillating terms appear for $t\to \infty$  only. 
 Singularities of $\bh(t)$ at points $t_{0} >0$ are also not excluded.
   
\section{Abstract operator theoretic input}

The main results of this section are Theorems~\ref{thm.b3}
and \ref{lma.c2}. They are used   for the proofs of the localization principle in Section~3 and of the symmetry principle in Section~4.

\subsection{Notation}
Here we recall some notation 
related to eigenvalues and singular values of compact operators. 
For a compact self-adjoint operator $A$, we denote by $\{\lambda_n^+(A)\}_{n=1}^\infty$ 
the non-increasing sequence of positive eigenvalues of $A$; we assume that 
the eigenvalues are enumerated with multiplicities and that the sequence 
$\{\lambda_n^+(A)\}_{n=1}^\infty$ is appended by zeros if $A$ has only finitely many 
positive eigenvalues. 
We also set $\lambda_n^-(A)=\lambda_n^+(-A)$. 
The singular values of a (not necessarily self-adjoint) compact operator
$A$ are defined by $s_n(A)=\lambda_n^+(\abs{A})$ 
where $\abs{A}=\sqrt{A^*A}$. 

To describe the  power asymptotics of the type \eqref{a3} and \eqref{a3a},  it is convenient to define the following functionals. 
 For $p>0$, 
and for a compact  operator $A$, we set
\begin{equation}
\Delta_p(A)
=
\limsup_{n\to\infty} ns_n(A)^p,
\quad
\delta_p(A)
=
\liminf_{n\to\infty} ns_n(A)^p.
\label{DD2}
\end{equation}
Moreover, if $A$ is self-adjoint, we denote
\begin{equation}
\Delta_p^\pm(A)
=
\limsup_{n\to\infty} n\lambda_n^\pm(A)^p,
\quad
\delta_p^\pm(A)
=
\liminf_{n\to\infty} n\lambda_n^\pm(A)^p.
\label{DD1}
\end{equation}
Put $A_\pm=(\abs{A}\pm A)/2$ so that $A_{\pm}=\pm\1_{\bbR_{\pm}}(A)A$.
 Since 
$ \lambda_n^\pm(A)= s_{n}(A_{\pm})$ for an arbitrary self-adjoint operator $A$,
 we have  
 \begin{equation}
 \Delta_p^\pm(A)=\Delta_p(A_\pm), \quad  \delta_p^\pm(A)=\delta_p(A_\pm).
 \label{DDD}
\end{equation}
In all concrete applications, our upper limits will coincide with the lower limits;
however, we work with the upper and lower limits separately 
  because it is more general and, at the same time,  it is technically more convenient.

We denote by $\Sch_{p,\infty}$ the class of all compact operators $A$ such that
$\Delta_p(A)$ is finite, and by $\Sch_{p,\infty}^0\subset\Sch_{p,\infty}$
the subclass of all operators $A$ such that $\Delta_p(A)=0$.
It is well known that both $\Sch_{p,\infty}$ and $\Sch_{p,\infty}^0$ are ideals of the algebra of bounded operators $\calB$;  in particular, they are linear spaces. 
Of course $A\in\Sch_{p,\infty}$ (or $A\in\Sch_{p,\infty}^0$) if and only if the same is true for its adjoint $A^*$.
We set $\Sch_0=\cap_{p>0}\Sch_{p,\infty}$, that is, 
\begin{equation}
A\in\Sch_0\quad \Leftrightarrow\quad s_n(A)=O(n^{-\alpha}),\quad n\to\infty,\quad \forall \alpha>0.
\label{eq:sss}\end{equation}

It is convenient to make use of the counting functions
$$ 
n(\lambda;A)=\#\{n: s_n(A)>\lambda\}, \quad n_\pm(\lambda;A)=\#\{n: \lambda_n^\pm(A)>\lambda\},
\quad \lambda>0.
$$
In terms of these functions, we have
\begin{equation}
\Delta_p(A)=\limsup_{\lambda\to0}\lambda^p n(\lambda;A), \quad
\Delta_p^\pm(A)=\limsup_{\lambda\to0}\lambda^p n_\pm(\lambda;A), 
\label{a4c}
\end{equation}
and similarly for the lower limits. 

Asymptotic formulas for     singular values and eigenvalues can be equivalently rewritten in terms of the functionals  \eqref{DD2} and  \eqref{DD1}. We make a 
standing assumption that the indices $\alpha>0$ and $p>0$ 
are related by $p=1/\alpha$. Then
$$
\lim_{n\to\infty }n^\alpha s_{n} (A)= c  \Longleftrightarrow \Delta_p (A)=\delta_p (A)=c ^p 
$$
and
$$
\lim_{n\to\infty }n^\alpha\lambda_{n}^\pm (A)= c^\pm \Longleftrightarrow \Delta_p^\pm(A)=\delta_p^\pm(A)=(c^\pm)^p .
$$
 
Since, for self-adjoint operators $A$,  the sequence $s_n(A)$ is the union of the two sequences
$\lambda^+_{n}(A)$ and $\lambda^-_{n}(A)$, for the counting functions we have
\begin{equation}
n(\lambda;A)=n_+(\lambda;A)+n_-(\lambda;A), \quad \lambda>0.
\label{a4b}
\end{equation}

\subsection{Asymptotically orthogonal operators}

First we recall a lemma which goes back to H.~Weyl.  
\begin{lemma}\label{lma.b1}\cite[Section~11.6]{BSbook}
Let $A$ be a compact operator and let $B\in\Sch_{p,\infty}^0$ for some $p>0$. 
Then 
\begin{equation}
\Delta_p(A+B)=\Delta_p(A)
\quad\text{ and }\quad
\delta_p(A+B)=\delta_p(A).
\label{b0a}
\end{equation}
If $A$ and $B$ are self-adjoint, then also 
$$
\Delta_p^\pm(A+B)=\Delta_p^\pm(A)
\quad\text{ and }\quad
\delta_p^\pm(A+B)=\delta_p^\pm(A).
$$
\end{lemma}
Recall the implication 
(see, e.g. \cite[Theorem 11.6.9]{BSbook})
\begin{equation}
A\in\Sch_{p,\infty}, 
\quad 
B\in\Sch_{p,\infty}
\quad 
\Rightarrow
\quad
A^*B\in\Sch_{p/2,\infty}, \quad
A B^*\in\Sch_{p/2,\infty}.
\label{b1}
\end{equation}
We will say that the operators $A$ and $B$ in $\Sch_{p,\infty}$ are asymptotically orthogonal
if the  class $\Sch_{p/2,\infty}$ in the right side of \eqref{b1} can be replaced by its subclass $\Sch_{p/2,\infty}^0$.
\begin{theorem}\label{thm.b2}\cite[Theorem~2.2]{III}
Let $A_1,\dots,A_L$ be compact operators such that for some $p>0$
$$
A_\ell^*A_j\in\Sch_{p/2,\infty}^0,
\quad 
A_\ell A_j^*\in\Sch_{p/2,\infty}^0
\quad 
\text{ for all $\ell\not=j$.}
$$
Then for   $A= A_1+\dots+A_L$, we have
$$
\Delta_p(A)\leq \sum_{\ell=1}^L \Delta_p(A_\ell), 
\quad
\delta_p(A)\geq \sum_{\ell=1}^L \delta_p(A_\ell).
$$
In particular, if $\Delta_p(A_\ell)=\delta_p(A_\ell)$ for all $\ell$, then 
$$
\Delta_p(A_\ell)=\delta_p(A_\ell)=
\sum_{\ell=1}^L\Delta_p(A_\ell).
$$
\end{theorem}
A very similar (only a slightly weaker) result was obtained much earlier in \cite{BS2}; 
note that the   proofs   in \cite{BS2} and \cite{III} are quite different.

Here we need an analogue of this statement for self-adjoint operators.
\begin{theorem}\label{thm.b3}
Let $A_1,\dots,A_L$ be compact self-adjoint operators such that 
for some $p>0$
\begin{equation}
A_\ell A_j\in\Sch_{p/2,\infty}^0
\quad 
\text{ for all $\ell\not=j$.}
\label{b4}
\end{equation}
Then for   $A= A_1+\dots+A_L$, we have
\begin{equation}
\Delta_p^\pm(A)\leq \sum_{\ell=1}^L \Delta_p^\pm(A_\ell), 
\quad
\delta_p^\pm(A)\geq \sum_{\ell=1}^L \delta_p^\pm(A_\ell).
\label{b5}
\end{equation}
In particular, if $\Delta_p^\pm(A_\ell)=\delta_p^\pm(A_\ell)$ for all $\ell$, then 
$$
\Delta_p^\pm(A)=\delta_p^\pm(A)=\sum_{\ell=1}^L \Delta_p^\pm(A_\ell).
$$
\end{theorem}

 Again, a version of this theorem can be found in \cite{BS2}; 
here we give a different proof. 
In order to explain the intuition behind this theorem, we observe that
the estimates \eqref{b5} are quite obvious if
 the operators $A_\ell$ are 
\emph{orthogonal} in the sense that 
\begin{equation}
A_j A_\ell=0, \quad \forall j\not=\ell.
\label{a8a}
\end{equation}
Then $A=A_1 +\dots+ A_L$ is a ``block-diagonal"
operator acting in the direct sum 
$\oplus_{\ell=1}^L \overline{\Ran (A_\ell)}$. It follows that
$$
n_\pm(\lambda; A)
=
\sum_{\ell=1}^L n_\pm(\lambda;A_\ell), 
\quad \lambda>0.
$$
Multiplying this by $\lambda^p$, taking $\lim\sup$ (resp. $\lim\inf$) 
as $\lambda\to0$
and recalling the expressions \eqref{a4c} for $\Delta_p^\pm$, $\delta_p^\pm$ 
in terms of the counting functions, 
we obtain the first (resp. the second) inequality in \eqref{b5}. 
Our goal is to replace the trivial condition \eqref{a8a} by a much weaker assumption   \eqref{b4}. 
 
In order to prove Theorem~\ref{thm.b3}, we will need the following 
auxiliary assertions.

\begin{lemma}\label{lma.b4}\cite[Proposition~4]{BS2}
Let $p>0$, and let $M_0$, $M_1$ be bounded non-negative self-adjoint operators  
such that $M_1-M_0\in \Sch_{p/2,\infty}^0$. Then 
$M_1^{1/2}-M_0^{1/2}\in\Sch_{p,\infty}^0$. 
\end{lemma}

The proof of this lemma in \cite{BS2} uses a non-trivial estimate of \cite[Theorem~3]{BKS}:
$$
\Delta^\pm_{p} (M_1^{1/2}-M_0^{1/2}) \leq c(p)
 \bigl( \Delta^\pm_{p/2} (M_1 -M_0 )\bigr)^{1/2}.
$$

\begin{lemma}\label{lma.b4x} 
Let $A$, $B$ be self-adjoint operators in $\Sch_{p,\infty}$ such that $A B\in\Sch_{p/2,\infty}^0$. Then 
\begin{equation}
\abs{A+B}-(\abs{A}+\abs{B})\in \Sch_{p,\infty}^0.
\label{8}
\end{equation}
\end{lemma}

\begin{proof}
Since 
\begin{equation}
\abs{A+B}^2-(\abs{A}+\abs{B})^2= AB+BA- |A| |B|+| B| |A|
\label{9}
\end{equation}
and
$
\abs{A}\abs{B}=\sign(A)AB\sign(B),
$
expression \eqref{9} belongs to $\Sch_{p/2,\infty}^0$.
So it remains to apply Lemma~\ref{lma.b4} with $M_0=(\abs{A}+\abs{B})^2$, $M_1=\abs{A+B}^2$.
\end{proof}

Adding and subtracting $A+B$ in \eqref{8}, we obtain

\begin{corollary}\label{lma.b4xx} 
Under the assumptions of Lemma~$\ref{lma.b4x}$, the inclusions 
\begin{equation}
(A+B)_\pm -(A_\pm +B_\pm)\in \Sch_{p,\infty}^0
\label{14}
\end{equation}
hold.
\end{corollary}

\subsection{Proof of Theorem~\ref{thm.b3}}

Using induction in $L$, it is easy to reduce the problem to the case $L=2$. 
Thus, changing our notation slighly, we will assume that 
$A$, $B$ are self-adjoint operators in $\Sch_{p,\infty}$ and $A B\in\Sch_{p/2,\infty}^0$.
We will prove that 
\begin{align}
\Delta_p^\pm(A+B)&\leq \Delta_p^\pm(A)+\Delta_p^\pm(B),
\label{7}
\\
\delta_p^\pm(A+B)&\geq \delta_p^\pm(A)+\delta_p^\pm(B).
\label{7a}
\end{align}

 According to the statement \eqref{b0a} of 
Lemma~\ref{lma.b1}, it follows from \eqref{14}  that
\begin{equation}
\Delta_p((A+B)_\pm)
=
\Delta_p(A_\pm+B_\pm).
\label{12}
\end{equation}
Since $A_\pm B_\pm \in \Sch_{p/2,\infty}^0$,   Theorem~\ref{thm.b2} implies that 
\begin{equation}
\Delta_p(A_\pm+B_\pm)
\leq
\Delta_p(A_\pm)+\Delta_p(B_\pm).
\label{13}
\end{equation}
Combining \eqref{12}, \eqref{13} and taking   \eqref{DDD} into account, we conclude the proof of \eqref{7}. 
  The estimate \eqref{7a} for the lower limits  can be obtained in a similar way. 
$\qed$

\subsection{Symmetry with respect to the reflection around zero}

The meaning of the following result is that if a compact self-adjoint operator $A$ is ``almost" unitarily equivalent to $-A$, then its positive and negative eigenvalues have the same asymptotic behaviour. 

\begin{theorem}\label{lma.c2}
Let $p>0$, and let $A$ be a compact self-adjoint operator such that for some 
unitary operator $U$, we have
\begin{equation}
R:=UAU^*+A\in\Sch_{p,\infty}^0.
\label{c6}
\end{equation}
Then 
\begin{align}
\Delta_p^+(A)&=\Delta_p^-(A)=\tfrac12\Delta_p(A),
\label{c7}
\\
\delta_p^+(A)&=\delta_p^-(A)=\tfrac12\delta_p(A).
\label{c7aa}
\end{align}
\end{theorem}

\begin{proof}
The first equalities in \eqref{c7}, \eqref{c7aa} are easy to check. 
Indeed, according to  Lemma~\ref{lma.b1} it follows  from \eqref{c6} that
\begin{align*}
\Delta_p^+(A)&=\Delta_p^+(UAU^*)=\Delta_p^+(-A)=\Delta_p^-(A),
\\
\delta_p^+(A)&=\delta_p^+(UAU^*)=\delta_p^+(-A)=\delta_p^-(A).
\end{align*}

The second pair of equalities is   more delicate. Let us multiply  \eqref{a4b}  by $\lambda^p$. Passing to the upper limit
as $\lambda\to 0$ and using the definition \eqref{a4c} of the quantities $\Delta_p$, we see that
\begin{equation}
\Delta_p(A)\leq \Delta_p^+(A)+\Delta_p^-(A)= 2 \Delta_p^+(A).
\label{De}
\end{equation}
Similarly,  passing to the lower limit, we see that
$$
\delta_p(A)\geq \delta_p^+(A)+\delta_p^-(A)= 2 \delta_p^+(A).
$$

It remains  to prove the opposite
estimates.  In view of  \eqref{a4b} for all $\lambda>0$, we have 
\begin{multline}
n(\lambda;A)
=
n_+(\lambda;A)+n_-(\lambda;A)
=
n_+(\lambda;A)+n_-(\lambda;UAU^*)
\\
=
n_+(\lambda;A)+n_-(\lambda;-A+R)
=
n_+(\lambda;A)+n_+(\lambda;A-R).
\label{c8}
\end{multline}
For any compact self-adjoint $A_1$, $A_2$ we have 
the inequality (see e.g. \cite[Theorem 9.2.9]{BSbook})
\begin{equation}
n_+(\lambda_1+\lambda_2;A_1+A_2)
\leq
n_+(\lambda_1;A_1)
+
n_+(\lambda_2;A_2),
\quad \lambda_1,\lambda_2>0.
\label{c9}
\end{equation}
In particular,  for every
 $\eps\in(0,1)$ 
 $$
 n_+( (1+\eps)\lambda ;A)\leq n_+(\lambda ;A-R) +n_+(\varepsilon\lambda ; R)
 $$
 and therefore \eqref{c8} yields the estimate
 $$
 n(\lambda;A)\geq n_+(\lambda;A)+ n_+((1+\varepsilon)\lambda;A)- n_+(\varepsilon\lambda ; R)
 \geq 2n_+((1+\varepsilon)\lambda;A)- n_+(\varepsilon\lambda ;R).
 $$
 Let us multiply  this estimate  by $\lambda^p$ and pass to the upper limit as $\lambda\to 0$. Since $R \in\Sch_{p,\infty}^0$, we have $\lambda^p n_+(\varepsilon\lambda ; R)\to 0$. Therefore according to  definition \eqref{a4c}, we find that
 $$
\Delta_p(A)\geq 2 (1+\varepsilon)^{-p} \Delta_p^+ (A).
$$
Hence $\Delta_p(A)\geq 2   \Delta_p^+ (A)$ because $\varepsilon>0$ is arbitrary. Together with \eqref{De}, this  proves 
 \eqref{c7}.

Similarly,  
\eqref{c9} implies  that for every
 $\eps\in(0,1)$ 
 $$
 n_+(  \lambda ;A-R)\leq n_+((1-\varepsilon)\lambda ;A) +n_+(\varepsilon\lambda ;-R)
 $$
and  therefore \eqref{c8} yields the estimate
 $$
 n(\lambda;A) \leq n_+( \lambda;A) +  n_+((1-\varepsilon)\lambda;A) + n_+(\varepsilon\lambda ; -R) \leq   2n_+((1-\varepsilon)\lambda;A) + n_+(\varepsilon\lambda ; -R).
 $$
  Let us again multiply it  by $\lambda^p$ and then pass to the lower limit whence
  $$
\delta_p(A)\leq 2 (1-\varepsilon)^{-p} \delta_p^+ (A).
$$
This suffices to conclude the proof of \eqref{c7aa}.
\end{proof}

\section{Localization principle}\label{sec.b}

The localization principle for eigenvalues of Hankel operators in the spaces $H^2(\bbT)$ and $H^2 (\bbR)$ will be stated in Theorems~\ref{thm.a3} and \ref{thm.b6}, respectively. 

\subsection{Hankel operators in  $H^2 (\bbT)$.}

In a standard way, we define   the singular support $\sing\supp\omega$ of a function 
$\omega\in L^\infty (\bbT)$   as the smallest closed set $X\subset \bbT$ 
such that $\omega\in C^\infty (\bbT\setminus X)$. 
Recall also  that the class $\Sch_0$ of compact operators was defined by \eqref{eq:sss}.

The key analytic ingredient of the proof of Therem~\ref{thm.a3} is the following statement.

\begin{lemma}\label{lma.b5}\cite[Lemma~2.6]{III}
Let $\omega_1,\omega_2\in L^\infty (\bbT)$ be such that 
$\sing\supp\omega_1\cap\sing\supp\omega_2=\varnothing$. 
Then 
$$
\Hank(\omega_1)^*\Hank(\omega_2)\in\Sch_0.
$$
\end{lemma}

Below we state the localization principle for eigenvalues of self-adjoint Hankel operators.
We use the functionals $\Delta_p^\pm$ and $\delta_p^\pm$, defined by formulas \eqref{DD1}.
\begin{theorem}\label{thm.a3}
Let $\omega_1,\dots,\omega_L\in L^\infty(\bbT)$ be symbols 
such that the singular supports of $\omega_\ell$ are disjoint:
$$
\sing\supp\omega_\ell\cap\sing\supp\omega_j=\varnothing,
\quad \ell\not=j,
$$
and such that  the symmetry condition \eqref{a2} is satisfied. 
Then for $\omega=\omega_1+\cdots+\omega_L$
and for an arbitrary $p>0$, we have
\begin{equation}
\Delta_p^\pm(H(\omega))\leq\sum_{\ell=1}^L\Delta_p^\pm(H(\omega_\ell)),
\quad
\delta_p^\pm(H(\omega))\geq\sum_{\ell=1}^L\delta_p^\pm(H(\omega_\ell)).
\label{a8}
\end{equation}
In particular, if 
$\Delta_p^\pm(H(\omega_\ell))=\delta_p^\pm(H(\omega_\ell))$
for all $\ell$, then 
$$
\Delta_p^\pm(H(\omega))=\delta_p^\pm(H(\omega))=\sum_{\ell=1}^L\Delta_p^\pm(H(\omega_\ell)).
$$
\end{theorem}
 
\begin{proof}
It suffices to use Theorem~\ref{thm.b3} with $A_\ell=\Hank(\omega_\ell)$. 
The inclusion $A_\ell A_j\in\Sch_{p/2,\infty}^0$ for $\ell\not=j$ follows from Lemma~\ref{lma.b5}.
 \end{proof}

Note that the expressions in \eqref{a8} may be infinite. 
In \cite{III}, we have an exact analogue of Theorem~\ref{thm.b3} 
for singular values of (not necessarily self-adjoint) Hankel operators.

\subsection{Hankel operators in $H^2(\bbR)$}
Hankel operators can also be defined  in the Hardy space $H^2(\bbR)$ 
of functions analytic in the upper half-plane.
We denote by $\Phi f=\wh f$ the  Fourier transform of $f$ in $L^2(\bbR)$, 
\begin{equation}
(\Phi f)(t)=
\wh f(t)=\frac1{\sqrt{2\pi}} \int_{-\infty}^\infty f(x) e^{-ixt}dx.
\label{b.ft}
\end{equation}
Let $H^2(\bbR)\subset L^2(\bbR)$ be the  Hardy class, 
$$
H^2(\bbR)=\{f\in L^2: \wh f(t)=0 \text{ for $t<0$}\},
$$
and let $\bP_+:L^2(\bbR)\to H^2(\bbR)$ be the corresponding orthogonal projection. 
Let $\bW$ be the involution in $L^2(\bbR)$, $(\bW f)(x)=f(-x)$. 
For ${\pmb\omega}\in L^\infty(\bbR)$, the Hankel operator 
$\bHank({\pmb\omega})$ in $H^2(\bbR)$ is defined by 
\begin{equation}
\bHank({\pmb\omega})f=\bP_+({\pmb\omega} \bW f), \quad f\in H^2(\bbR).
\label{b6}
\end{equation}
It is straightforward to see that the symmetry condition 
$$
\overline{\bomega(x)}=\bomega(-x), \quad x\in\bbR,
$$
ensures that $\bHank(\bomega)$ is self-adjoint.

There is a unitary equivalence between the  Hankel operators $\Hank(\omega)$ 
defined in $H^2(\bbT)$ by formula \eqref{a1} 
and the Hankel operators $\bHank({\pmb\omega})$ defined 
in $H^2(\bbR)$ by formula \eqref{b6}. 
Indeed, let
\begin{equation}
w=\frac{z-i/2}{z+i/2}, \quad z=\frac{i}{2}\frac{1+w}{1-w}, 
\label{b7}
\end{equation}
be the  standard conformal map sending the upper half-plane
 onto the unit disc, and let  $\calU:L^2(\bbT)\to L^2(\bbR)$ be
 the corresponding unitary operator defined by
$$
(\calU f)(x)
=
\tfrac1{\sqrt{2\pi}}\tfrac1{x+i/2} f(\tfrac{x-i/2}{x+i/2}), 
\quad
(\calU^* \mathbf f)(\mu)
=
i\sqrt{2\pi}\tfrac1{1-\mu}\mathbf  f (\tfrac{i}{2}\tfrac{1+\mu}{1-\mu}).
$$
Then
\begin{equation}
 \calU \Hank(\omega) \calU^*=\bHank({\pmb\omega})
\label{b8}
\end{equation}
provided
\begin{equation}
 {\pmb\omega} (x)= -\tfrac{x-i/2}{x+i/2}\omega (\tfrac{x-i/2}{x+i/2}).
\label{b9}
\end{equation}

Symbols $\bomega (x)$ of Hankel operators \eqref{b6}
   have the exceptional point $x=\infty$. 
   In order to rewrite the results obtained for Hankel operators $H(\omega)$ in terms of the
   Hankel operators ${\bf H}(\pmb\omega)$,
   we identify the points $x=+\infty$ and $x=-\infty$. The real line with such identification will be denoted $\bbR_{*}$. We write $\pmb\omega \in C (\bbR_{*})$ if $\pmb\omega \in C (\bbR )$ and if 
$$
\lim_{x\to \infty}\pmb\omega (x)=\lim_{x\to - \infty}\pmb\omega(x).
$$
Similarly, we write $\pmb\omega\in C^\infty (\bbR_{*})$ if 
$\pmb\omega\in C^\infty (\bbR )$ and 
\begin{equation}
\lim_{x\to \infty}\pmb\omega^{(m)} (x)=\lim_{x\to - \infty}\pmb\omega^{(m)}(x).
\label{b11}
\end{equation}
In particular, the point $x=\infty$ belongs to the singular support of $\pmb\omega (x)$ if 
for some $m\geq0$ the relation \eqref{b11} fails
(i.e. if either at least one of the limits does not exist or if the limits are not equal).

In view of relations 
\eqref{b8} and \eqref{b9} the localization principle for Hankel operators in $H^2(\bbR)$ given below is a direct consequence of the localization principle in $H^2(\bbT)$ (Theorem~\ref{thm.a3}). 

\begin{theorem}\label{thm.b6}
Let $\bomega_1,\dots,\bomega_L\in C(\bbR_*)$ be symbols 
satisfying the symmetry condition 
\begin{equation}
\overline{\bomega(x)}=\bomega(-x)
\label{eq:HS}
\end{equation}
  and such that 
the singular supports of $\bomega_\ell$ for different $\ell$ are disjoint.  
Then for the symbol $\bomega=\bomega_1+\cdots\bomega_L$ and for any $p>0$ we have
$$
\Delta_p^\pm(\bHank(\bomega))\leq\sum_{\ell=1}^L\Delta_p^\pm(\bHank(\bomega_\ell)),
\quad
\delta_p^\pm(\bHank(\bomega))\geq\sum_{\ell=1}^L\delta_p^\pm(\bHank(\bomega_\ell)).
$$
In particular, if $\Delta_p^\pm(\bHank(\bomega_\ell))=\delta_p^\pm(\bHank(\bomega_\ell))$ 
for all $\ell$, then 
$$
\Delta_p^\pm(\bHank(\bomega))=\delta_p^\pm(\bHank(\bomega))
=
\sum_{\ell=1}^L\Delta_p^\pm(\bHank(\bomega_\ell)).
$$
\end{theorem}

\section{Symmetry principle}\label{sec.c}

The symmetry  principle for the eigenvalues of Hankel operators in the spaces $H^2(\bbT)$ and $H^2(\bbR)$ will be stated in Theorems~\ref{thm.a1} and \ref{thm.c3}, respectively, in terms of the functionals $\Delta_p^\pm$, $\delta_p^\pm$ (see \eqref{DD2} and \eqref{DD1}). Moreover, in Section~4.3 we discuss the symmetry  principle for the essential spectrum of Hankel operators.

\subsection{Symmetry principle in $H^2(\bbT)$}

The symmetry principle for compact self-adjoint Hankel operators $\Hank(\omega)$ in the space $L^2(\bbT)$ can be stated as follows.

\begin{theorem}\label{thm.a1}
Let $\omega\in L^\infty (\bbT)$ be a symbol satisfying the symmetry condition $\eqref{a2}$
and such that $\sing\supp\omega$ does not contain the points $1$ and $-1$. 
Then for any $p>0$, 
\begin{align*}
\Delta_p^+(\Hank(\omega))&=\Delta_p^-(\Hank(\omega))=\tfrac12\Delta_p(\Hank(\omega)),
\\
\delta_p^+(\Hank(\omega))&=\delta_p^-(\Hank(\omega))=\tfrac12\delta_p(\Hank(\omega)).
\end{align*}
\end{theorem}

If $\sing\supp\omega$ contains the points $1$ or $-1$, the symmetry breaks down:
the contribution of each of these points to the spectrum is not symmetric. 
This will be illustrated below by Theorem~\ref{HT}.

Put
\begin{equation}
s(\mu)=\sign \Im\mu, \quad \mu\in \bbT.
\label{Ss}
\end{equation}
The operator of 
multiplication by $s$ in $L^2(\bbT)$ will also be denoted by $s$. 
Clearly, $s=s^*$  and $s^2=I$.
 In the following statement, $\Hank(\omega)$ is 
not necessarily self-adjoint. 

Below it will be  convenient to consider Hankel operators $H(\omega)$ as operators 
acting not on the Hardy class, but on the space $L^2 (\bbT)$; in this case $H(\omega)$ is defined
by the formula
\begin{equation}
\Hank(\omega)=P_+\omega WP_+. 
\label{L2}
\end{equation}
Of course, the non-zero spectra of the operators \eqref{a1}  and \eqref{L2} coincide.

\begin{lemma}\label{lma.c1}
Let $\omega\in L^\infty(\bbT)$ be such that the singular support of $\omega$ 
does not contain the points $1$ and $-1$. 
Then 
\begin{equation}
s\Hank(\omega) +\Hank(\omega)s \in\Sch_0.
\label{c1}
\end{equation}
\end{lemma}

\begin{proof}
We will use two well-known facts (see the book \cite{Peller} and Lemma~4.2 in \cite{III}, for additional details):

(i) if $\sigma\in C^\infty(\bbT)$, then $\Hank(\sigma)\in\Sch_0$;

(ii) if $\sigma\in C^\infty(\bbT)$, then the commutator $[\sigma,P_+] : =\sigma P_+ -P_{+} \sigma\in\Sch_0$.

Write $\omega=\omega_0+\omega_1$, where 
$\omega_0 \in L^\infty(\bbT)$ vanishes
identically in a  neighborhood of $\{-1,1\}$, and $\omega_1\in C^\infty(\bbT)$. 
By (i), it suffices to prove \eqref{c1} with $\omega_0$ instead of $\omega$. 
In what follows, we drop the subscript $0$ and simply assume that $\omega$ vanishes
in a  neighborhood of $\{-1,1\}$.

Put $\omega_*(\mu)=\omega(\overline\mu)$.  
Let us choose $\varphi\in C^\infty(\bbT)$ such that 
$\varphi\omega=\omega$ and $\varphi\omega_*=\omega_*$, 
and $\varphi$ vanishes in a  neighborhood of $\{-1,1\}$. Then we also have $s\varphi\in C^\infty(\bbT)$.
It follows from (ii)   that
$$
[s,P_+] \varphi
=
sP_+ \varphi -P_+s \varphi
=
s[P_+, \varphi]+[s \varphi,P_+]
\in\Sch_0,
$$
whence
\begin{equation}
sP_+\omega  
-
P_+s\omega  
=
[s,P_+]\omega  
=
[s,P_+] \varphi \omega  \in\Sch_0
\label{c3}
\end{equation}
and, multiplying by $WP_+$ on the right, 
\begin{equation}
s H(\omega)  
-
P_+ s\omega  WP_{+}
   \in\Sch_0 .
\label{c3x}
\end{equation}
Similarly to \eqref{c3}, we have
$$
\omega_*P_+ s
-
 \omega_* s P_+\in\Sch_0.
$$
Therefore using that $\Hank(\omega)=P_+ W\omega_{*} P_+ $ and multiplying by $P_+W$ on the left,   we    obtain
\begin{equation}
 H(\omega) s 
-
P_+ W \omega_{*} s   P_{+}
   \in\Sch_0.
\label{c3x1}
\end{equation}
Putting together \eqref{c3x},  \eqref{c3x1} and taking into account that $s\omega  W+W \omega_{*} s=0$, we conclude the proof of  \eqref{c1}.
\end{proof}

\begin{proof}[Proof of  Theorem~\ref{thm.a1}]
It remains to use Theorem~\ref{lma.c2} with $A=H(\omega)$ and $U=s$;
the inclusion \eqref{c6} in the hypothesis of this theorem holds true by 
Lemma~\ref{lma.c1}.
 \end{proof}

\subsection{Symmetry principle in $H^2(\bbR)$}

  The symmetry principles in $H^2(\bbT)$ and $H^2(\bbR)$ are equivalent.

\begin{theorem}\label{thm.c3}
Let $\bomega\in L^\infty (\bbR)$ satisfy the symmetry relation $\overline{\bomega(x)}=\bomega(-x)$,
and suppose that $\sing\supp\bomega$ does not contain $0$ or $\infty$.  
Then for any $p>0$, 
\begin{align}
\Delta_p^+(\bHank(\bomega))&=\Delta_p^-(\bHank(\bomega))=\tfrac12\Delta_p(\bHank(\bomega)),
\label{c15}
\\
\delta_p^+(\bHank(\bomega))&=\delta_p^-(\bHank(\bomega))=\tfrac12\delta_p(\bHank(\bomega)).
\label{c16}
\end{align}
\end{theorem}

\begin{proof}
Observe that the map \eqref{b7}  sends the point $z=\infty$ into the point $w= 1$ and the point $z=0$ into the point $w= -1$. Let the symbol $\omega$ be defined by formula  \eqref{b9}. Since its singular support does not contain the points $1$ and $-1$, Theorem~\ref{thm.a1} applies to the Hankel operator $H(\omega)$. According to \eqref{b8} the operators $H(\omega)$ and $\bHank(\bomega)$ are unitarily equivalent, which yields \eqref{c15} and \eqref{c16}.
\end{proof}

\subsection{Essential spectrum}

Although this is not the focus of the present paper, we mention that 
some variants of the symmetry principle also hold true for non-compact Hankel operators. 
For example, we have

\begin{theorem}\label{prop.ess}
Let $\omega\in L^\infty(\bbT)$ be a symbol satisfying the  symmetry condition $\eqref{a2}$. Suppose that $\omega$ is continuous in  some neighborhoods of the points $1$ and $-1$. Then 
\begin{equation}
\sigma_\ess(H(\omega))=\sigma_\ess(-H(\omega)).
\label{eq:ess}
\end{equation}
\end{theorem}
\begin{proof} 
Let $\omega$ be continuous on the union $G$ of two arcs $[e^{-i\delta}, e^{i \delta}]$ and $[-e^{i\delta}, -e^{-i\delta}]$ for some $\delta>0$.
There exist functions $\omega_n\in C^\infty(G)$ such that $\| \omega-\omega_{n}\|_{L^\infty(G)}\to 0$ as $n\to \infty$. We set
$\omega_n (\mu)= \omega  (\mu)$ for $\mu\in\bbT\setminus G$. Then $\| \omega-\omega_{n}\|_{L^\infty(\bbT)}\to 0$ as $n\to \infty$. It follows that  $\norm{H(\omega)-H(\omega_n)}\to 0$ and hence
$$
\norm{(s H(\omega)s+H(\omega))-(sH(\omega_n)s+H(\omega_n))}\to0,
\quad
n\to\infty,
$$
where $s$ is defined by formula \eqref{Ss}.
By Lemma~\ref{lma.c1}, the operators $sH(\omega_n)s+H(\omega_n)$ are compact for all $n$ and so
  the operator $sH(\omega)s+H(\omega)$ is also compact. 
Applying H.~Weyl's theorem on the stability of the essential spectrum under compact
perturbations, we obtain that
$$
\sigma_\ess(H(\omega))=\sigma_\ess(sH(\omega)s)=\sigma_\ess(-H(\omega)),
$$
as required. 
\end{proof}

We are not aware of this statement appearing explicitly in the literature, although
similar considerations have been used by S.~Power in his work \cite{Power}.

If $\omega$ is discontinuous at $1$ or $-1$, then in general the symmetry \eqref{eq:ess}  breaks down (see formula \eqref{eq:XX4}).

Of course Theorem~\ref{prop.ess} can be reformulated in terms of Hankel operators in the space $H^2 (\bbR)$.

\begin{theorem}\label{essc}
Let $\bomega\in L^\infty(\bbR)$ be a symbol satisfying the   condition $\eqref{eq:HS}$. Suppose that $\bomega$ is continuous in  
neighborhoods of the points $0$ and $\infty$. Then 
$$
\sigma_\ess(\bHank(\bomega))=\sigma_\ess(- \bHank(\bomega)).
$$
\end{theorem}

\section{Spectral asymptotics for Hankel operators in $\ell^2(\bbZ_+)$}\label{sec.d}

Recall that Hankel operators $\Gamma(h)$  in $\ell^2(\bbZ_+)$ were defined  by formula \eqref{a11}.
The main result of this section is Theorem~\ref{thm.a5}.  
It  gives the asymptotics
of eigenvalues of   operators $\Gamma(h)$ 
corresponding to
``oscillating"  sequences $h$ 
of the form \eqref{eq:AS}. 
An equivalent result for 
Hankel operators $H(\omega)$ in the Hardy space $H^2(\bbT)$ is stated in Theorem~\ref{HT}.

\subsection{Previous results}
We proceed from a particular case of Theorem~\ref{thm.a5} when 
the asymptotics of $h(j)$ consists of one term only.

\begin{theorem}\label{thm.a6}\cite[Theorem~1.1]{II}
Let $\alpha>0$ and let  
\begin{equation}
q(j)=
j^{-1}(\log j)^{-\alpha},
\quad
j\geq2,
\label{a24}
\end{equation}
$($the choice of any finite number of terms of the sequence $q$
is not important$)$.
Then the eigenvalues of the Hankel operator $\Gank(q)$ 
satisfy the asymptotic relation 
$$
\lambda_n^+(\Gank(q))= 
\varkappa(\alpha)n^{-\alpha} + o(n^{-\alpha}),
\quad 
\lambda_n^-(\Gank(q))= o(n^{-\alpha}),
$$
where the coefficient $\varkappa(\alpha)$ is given by formula \eqref{a15}.
\end{theorem}

 Let $q_{- 1}(j)= (-1)^j q(j)$, and let the unitary operator $T$ 
 in $\ell^2(\bbZ_+)$
 be defined by the relation $(T u)(j)=(-1)^j u(j)$. Then $\Gank(q_{- 1})=T^* \Gank(q) T$    whence $\lambda_{n}^\pm (\Gank(q_{- 1}))= \lambda_{n}^\pm (\Gank(q))$. Therefore Theorem~\ref{thm.a6} yields

\begin{corollary}\label{cor.a6} 
The conclusions of Theorem~$\ref{thm.a6} $ are true for the sequence 
$$ q_{- 1}(j)= (-1)^j
j^{-1}(\log j)^{-\alpha}.$$
\end{corollary}

The first step in the proof of  Theorem~\ref{thm.a5} is to construct a symbol 
corresponding to the sequence \eqref{a24}. It is convenient to consider a slightly more general case.

\begin{lemma}\label{lma.d1}\cite[Lemma~4.3]{III}
Let $\zeta\in\bbT$, $\alpha\geq 0$, and let  
\begin{equation}
q_{\zeta}(j)=\zeta^{-j} j^{-1}(\log j)^{-\alpha}, \quad  j\geq2,
\label{symb}
\end{equation}
$(q_{\zeta}(0)=q_{\zeta} (1)=0)$. Put
\begin{equation}
\omega_{\zeta}(\mu)=\sum_{j=2}^\infty j^{-1}(\log j)^{-\alpha} \bigl( ( \mu/\zeta)^j- ( \mu/\zeta)^{-j}\bigr),
\quad
\mu\in\bbT.
\label{symb1}
\end{equation}
Then $\omega_{\zeta} \in L^\infty(\bbT)$ and $\omega_{\zeta}\in C^\infty(\bbT \setminus\{ \zeta\})$. 
For the Fourier coefficients
 of function \eqref{symb1}, we have   $\wh\omega_{\zeta}(j) =q_{\zeta}(j)$ for all $j\geq 0$.
\end{lemma}

The assertion below is a particular case of our general result (Theorem~3.1 in \cite{III}) on the asymptotics of singular values of Hankel operators, needed in the present text.
Its proof in \cite{III} uses the localization principle for singular values (which is the analogue of Theorem~\ref{thm.a3}).

\begin{theorem}\label{thm.a4}\cite[Theorem~3.1]{III}
Let $\alpha>0$, let $\zeta \in\bbT$, $\Im \zeta\neq 0$, and let  $b \in \bbC$ be arbitrary. 
Consider the sequence $h$ given by  $h(0)=h(1)=0$ and
$$
h(j)=  2 \Re ( b  \zeta ^{-j}) j^{-1}(\log j)^{-\alpha}, 
\quad 
j\geq2.
$$
Then the singular values of $\Gank(h)$ satisfy the asymptotic relation 
\begin{equation}
s_n(\Gank(h))= 
2^\alpha \varkappa(\alpha)|b| n^{-\alpha} + o(n^{-\alpha}).
\label{a14}
\end{equation}
\end{theorem}

In view of the symmetry principle,  this results yields 
the asymptotics of the eigenvalues of $\Gank(h)$.

\begin{theorem}\label{CC} 
Let $\Gank(h)$ be the same as in Theorem~$\ref{thm.a4}$. 
Then  
\begin{equation}
\lambda^\pm_n(\Gank(h))= 
  \varkappa(\alpha) |b| n^{-\alpha} + o(n^{-\alpha}).
\label{eq:CC}
\end{equation}
\end{theorem}

\begin{proof}
Let the symbol $\omega_\zeta$ be defined by formula \eqref{symb1} and let $\omega=b\omega_\zeta+\overline{b}\omega_{\overline{\zeta}}$. Then $\wh{\omega} (j) = h(j)$ for $j\geq0$, and hence the operators
  $\Gamma(h)$ and $H(\omega)$ are unitarily equivalent.
 By Lemma~\ref{lma.d1}, the singular support of  $\omega$ consists of the pair of points $\zeta,\overline{\zeta}$.  
Therefore by the symmetry principle
(Theorem~\ref{thm.a1}) we have 
$$
\Delta_p^+(\Gamma(h))=\Delta_p^-(\Gamma(h))=\tfrac12\Delta_p(\Gamma(h)),
$$
and similarly for the lower limits. The asymptotic relation \eqref{a14} for the singular values 
can be equivalently rewritten as 
$\Delta_p(\Gamma(h))=2\varkappa(\alpha)^p |b|^p$, and thus we obtain
$$
\Delta_p^+(\Gamma(h))=\Delta_p^-(\Gamma(h))=\varkappa(\alpha)^p |b|^p
$$
and similarly for the lower limits. This yields \eqref{eq:CC}.
\end{proof}

In order to estimate the error terms, we use the following result of \cite{I}.  
Let $[\alpha]$ be the integer part of $\alpha$, 
$[\alpha]=\max\{m\in\bbZ: m\leq\alpha\}$.
We set 
\begin{equation}
M(\alpha)=
\begin{cases}
[\alpha]+1, & \text{ if } \alpha\geq1/2,
\\
0, & \text{ if } \alpha<1/2.
\end{cases}
\label{a16}
\end{equation}
For a sequence $g=\{g(j)\}_{j=0}^\infty$, we define iteratively the sequences
$g^{(m)}=\{g^{(m)}(j)\}_{j=0}^\infty$, $m=0,1,2,\dots$, by setting $g^{(0)}(j)=g(j)$ for all $j$ 
and 
$$
g^{(m+1)}(j)=g^{(m)}(j+1)-g^{(m)}(j), 
\quad j\geq0.
$$ 

Before stating the next result, let us comment that  for the sequence  $q$ defined by \eqref{a24},
the sequences $q^{(m)}$ for all $m\geq1$   satisfy 
$$
q^{(m)}(j)=O(j^{-1-m}(\log j)^{-\alpha}), \quad j\to\infty.
$$
On the other hand,   the sequence    \eqref{symb} with $\zeta\neq 1$ satisfies  only the condition
$q_{\zeta}^{(m)}(j)=O(j^{-1}(\log j)^{-\alpha})$ for any $m\geq1$. Nevertheless we have the following assertion.

\begin{theorem}\label{thm.a5a}\cite[Theorem~2.3]{I}
Let $\alpha>0$ and  let $M=M(\alpha)$ be the integer given by \eqref{a16}.
Let $g$ be a complex valued sequence such that
\begin{equation}
g^{(m)}(j)=o( j^{-1-m}(\log j)^{-\alpha}), \quad 
j\to\infty, 
\label{a16a}
\end{equation}
  for all $m=0,\dots,M$. Pick
 any   $\zeta \in\bbT$ and put $g_{\zeta} (j)= \zeta^{-j} g(j)$.
Then $ s_n(\Gank( g_{\zeta} ))=   o(n^{-\alpha})$.
\end{theorem}

\subsection{Asymptotics of eigenvalues}

Our main result below concerns the real sequences of the form \eqref{eq:AS}.

\begin{theorem}\label{thm.a5}
Let $\alpha>0$, $p=1/\alpha$; let $\varphi_1,\dots,\varphi_L\in(0,\pi)$ be 
distinct numbers and let $\psi_1,\dots\psi_L\in \bbR$ as well as
${\sf b}_{1},{\sf b}_{-1}\in\bbR$, 
$b_1,\dots,b_L\in\bbR$ be arbitrary. 
Let $h$ be a sequence of real numbers such that 
\begin{multline}
h(j)=
{\sf b}_1j^{-1}(\log j)^{-\alpha}
+ {\sf g}_{1}(j) + (-1)^{j} \bigl(
{\sf b}_{-1} j^{-1} (\log j)^{-\alpha} + {\sf g}_{-1}(j) \bigr)
\\
+
2\sum_{\ell=1}^L \bigl(b_\ell j^{-1} (\log j)^{-\alpha}+g_{\ell}(j)\bigr)\cos(\varphi_\ell j-\psi_\ell),
\quad 
j\geq2,
\label{a17}
\end{multline}
where all error terms ${\sf g}_{1}, {\sf g}_{-1},g_{1}, \dots, g_{L}$   
satisfy condition  \eqref{a16a} for all $m=0,1,\dots,M (\alpha)$ $(M(\alpha)$ is given by \eqref{a16}$)$. 
Then the eigenvalues of the Hankel operator $\Gank(h)$ 
satisfy the asymptotic relation \eqref{a3}
with the coefficients $a^\pm$  defined by \eqref{a19}. 
\end{theorem}

\begin{proof}
It is convenient to give the proof in terms of the functionals $\Delta_p^\pm$, $\delta_p^\pm$, 
(see \eqref{DD1} or \eqref{a4c}).
We first consider every term in the right-hand side of \eqref{a17} separately. Put
$$
{\sh}_1(j)={\sf b}_{1} j^{-1}(\log j)^{-\alpha}, 
\quad
{\sh}_{-1}(j)={\sf b}_{-1}(-1)^{j}j^{-1}(\log j)^{-\alpha}, 
$$
and for $\ell=1,\dots,L$,
$$
h_\ell(j)
=
2b_\ell\cos(\varphi_\ell j-\psi_\ell) j^{-1} (\log j)^{-\alpha}
=
2\Re(b_\ell e^{i\psi_\ell}\zeta_\ell^{-j})j^{-1} (\log j)^{-\alpha}, 
$$
where $\zeta_\ell=e^{i\varphi_\ell}$. 
By Theorem~\ref{thm.a6} and  Corollary~\ref{cor.a6}, we have
\begin{equation}
\begin{split}
\Delta_p^\pm(\Gank({\sh}_{1}))
=
\delta_p^\pm(\Gank({\sh}_{1}))
&=
 ( \varkappa(\alpha) {\sf b}_{1})_\pm^p,
\\
\Delta_p^\pm(\Gank({\sh}_{-1}))
=
\delta_p^\pm(\Gank({\sh}_{-1}))
&=
 ( \varkappa(\alpha){\sf b}_{-1})_\pm^p,
\label{d1}
\end{split}
\end{equation}
and by Theorem~\ref{CC}, we have
\begin{equation}
\Delta_p^\pm(\Gank(h_\ell))= 
\delta_p^\pm(\Gank(h_\ell))= 
( \varkappa(\alpha)\abs{b_\ell})^p , \quad \ell=1,\ldots, L.
\label{d3}
\end{equation}
It follows from Lemma~\ref{lma.d1} that the singular supports of the symbols of the operators $\Gamma ({\sh}_{1})$, $\Gamma ({\sh}_{-1})$ and
$\Gamma (h_{\ell})$ consist of the points $1$, $-1$ and of the pairs $\zeta_{\ell}$, $\overline{\zeta_{\ell}}$, respectively.
So  we can apply Theorem~\ref{thm.a3} (the localization principle for eigenvalues) to the Hankel operator $\Gamma(h_*)$ with
$$
h_*= {\sh}_1+{\sh}_{-1}+\sum_{\ell=1}^L h_\ell
$$
which yields
$$
\Delta_p^\pm(\Gank(h_*))
=
\delta_p^\pm(\Gank(h_*))
=
\Delta_p^\pm(\Gank({\sh}_{1}))
+
\Delta_p^\pm(\Gank({\sh}_{-1}))
+
\sum_{\ell=1}^L \Delta_p^\pm(\Gank(h_\ell)).
$$
Now relations   \eqref{d1} and \eqref{d3} imply that
\begin{equation}
\Delta_p^\pm(\Gank(h_*))
=\delta_p^\pm(\Gank(h_*))
=\varkappa(\alpha)^p\bigl(({\sf b}_{1})_\pm^p+({\sf b}_{-1})_\pm^p+\sum_{\ell=1}^L \abs{b_\ell}^p\bigr).
\label{YZ}
\end{equation}

Finally, set $g=h -h_{*}$. Using the representation \eqref{a17} and our conditions on  ${\sf g}_{1}, {\sf g}_{-1},g_{1}, \dots, g_{L}$ and applying Theorem~\ref{thm.a5a}, we see that $ \Gank(g)\in\Sch_{p,\infty}^0$. Since
$\Gank(h)=\Gank(h_*)+\Gank(g)$, it follows from   Lemma~\ref{lma.b1}  that 
$$
\Delta_p^\pm(\Gank(h))=\Delta_p^\pm(\Gank(h_*))\quad 
\text{and} 
\quad \delta_p^\pm(\Gank(h))=\delta_p^\pm(\Gank(h_*)).
$$ 
Thus by \eqref{YZ}, we obtain the relations \eqref{a3}, \eqref{a19}. 
\end{proof}

\subsection{Spectral asymptotics for Hankel operators in the Hardy space}

Here we give an analogue of Theorem~\ref{thm.a5} in terms of the Hankel operators $H(\omega)$ in the space $H^2(\bbT)$. They are linked to the operators $\Gamma (h)$ by formulas \eqref{a9}, \eqref{a12}. 
Below we consider a class of symbols $\omega$ whose Fourier coefficients satisfy the asymptotic relation \eqref{a17}. 
All necessary calculations have  already been done in \cite{IV}. Here we only state the results.  Note that our notation is slightly different from that in \cite{IV} because in \cite{IV} Hankel operators were considered in a different representation.

We consider a class of functions $\omega(\mu)$ that are smooth on the unit circle except at some finite number of points   where they have logarithmic singularities. We describe an admissible singularity supposing first that it is located at the point $\mu=1$.  
Let us introduce an even function 
$\chi_0\in C^\infty(\bbR)$  satisfying the condition
$$
\chi_0(\theta)=
\begin{cases}
1& \text{for $\abs{\theta}\leq c_1$,}
\\
0& \text{for  $\abs{\theta}\geq c_2$,}
\end{cases}
$$
with $c_1\in (0,c_2)$ and sufficiently small  $c_2$. We accept the following sufficiently general assumption.

\begin{assumption}\label{AsLog}
Let $\alpha>0$, and let $v_{j, \sigma}(\theta)$ and $u_{j, \sigma}(\theta)$, $j=0,1$, $\sigma=\pm$, be complex valued $C^\infty$ functions of $\theta\in \bbR$ such that
\begin{equation}
 v_{0,+}  (0) =v_{0,-}  (0)=: v_{0} .
 \label{eq:bbx}
\end{equation}
Then the function $\omega$ is defined by  the relation
\begin{equation}
\omega(e^{i\theta})
=  
\sum_{j=0,1}\sum_{\sigma=\pm} 
v_{j, \sigma} (\theta)(-\log\abs{\theta} +  u_{j, \sigma }(\theta))^{1-j-\alpha}\1_\sigma(\theta) \chi_0(\theta), 
\quad \theta\in(-\pi,\pi].
\label{b4cx}
\end{equation}
\end{assumption}

Here $c_{2}$ is chosen so small   that $\theta=0$ is the only singularity of  function \eqref{b4cx},
that is, 
$$-\log\abs{\theta} +  u_{j, \sigma }(\theta)\neq 0  \quad \text{if} \quad \theta\in [-c_{2}, c_{2}]  
$$
for $ j=0,1$, $\sigma=\pm$. 
The branch of the function $z^{j-\alpha} =e^{(j-\alpha)\log z}$ where $z=-\log\abs{\theta} +  u_{j,\sigma} (\theta)$ 
 is fixed by the condition 
$$
\arg (-\log\abs{\theta} +  u_{j,\sigma} (\theta))\to 0 \quad \text{as}  \quad \theta\to 0.
$$

We emphasize that because of the additional factor $\log|\theta|$, the terms in \eqref{b4cx} corresponding to $j=0$
  are more singular than the terms   corresponding to $j=1$. However due to the condition  \eqref{eq:bbx} the sum  of the terms with $j=0$ over $\sigma=+,-$ is essentially an even function of $\theta$. It can be deduced from this fact that the contribution of this sum to the asymptotics of the Fourier coefficients is of the same order as that of the terms     corresponding to $j=1$. 

For a function $\omega$ satisfying Assumption~\ref{AsLog}, we put
\begin{equation}
b=
(1-\alpha) v_0 \bigl( \tfrac12 + \tfrac1{2\pi i}  (u_{0,+ } (0)- u_{0,- } (0)) \bigr)
+    \tfrac1{2\pi i} (v_{1,+}  (0) - v_{1,-} (0) ) .
\label{b4ab}
\end{equation}
If $\omega(\overline\mu)=\overline{\omega(\mu)}$, then it follows from equality \eqref{b4cx}  that necessarily
$$
v_0 = \overline{v}_0,   \quad
u_{0,+ } (0)= \overline{u_{0,- } (0)}, \quad    v_{1,+ } (0)=\overline{v_{1,- } (0)} .
$$
In  this case $b= {\sf b} $ where
 \begin{equation}
{\sf b} 
=
(1-\alpha)v_0 \bigl( \tfrac12 +  \tfrac1\pi \Im u_{0,+ }(0)   \bigr)
+    \tfrac1\pi\Im v_{1,+}  (0)  , \quad v_{0}=\overline{v}_{0}, 
\label{b4abS}
\end{equation}
is real.

From the analytic point of view we rely on the following assertion.  

\begin{theorem}\label{lma.be}\cite[Theorem 3.2]{IV}
Under Assumption~$\ref{AsLog}$,  the Fourier coefficients $\wh\omega (j)$ of $\omega (\mu)$ admit  the   representation
\begin{equation}
\wh\omega ( j)
=
b    j^{-1}(\log j)^{-\alpha}+g( j), 
\label{eq:diff}
\end{equation}
where the coefficient $b$ is given by formula \eqref{b4ab} and the error term $g (j)$ satisfies the estimates 
$$
g ^{(m)} ( j)=O \bigl(j^{-1-m}(\log j)^{-\alpha-1}\bigr), \quad j \to \infty,
$$
for all $m\geq0$. 
\end{theorem}

Note that in \cite{IV} the asymptotics of  $\wh\omega (j)$ was considered for $j\to-\infty$. 
In order to translate the results of \cite{IV} into the context of this paper, one needs to 
use the complex conjugation: $\wh\omega_{1} (-j) = \overline{\wh\omega (j)}$ if $ \omega_{1} (\mu) = \overline{ \omega (\mu)}$.

 We emphasize that
the leading term of the asymptotics of the Fourier coefficients of the function \eqref{b4cx} 
depends on the combination \eqref{b4ab} only.  We also note that without condition \eqref{eq:bbx} asymptotics of $\wh\omega ( j)$ would be different from  \eqref{eq:diff}.

Here we state a result about the eigenvalue   asymptotics for self-adjoint Hankel operators $H(\omega)$ with
symbols    having finitely many 
logarithmic singularities. Thus we suppose that the symbol is a sum of the functions $\omega_{\ell}(\mu/\zeta_{\ell})$ where $\zeta_{\ell}$ are distinct points of $\bbT$ and each $\omega_{\ell}$ satisfies Assumption~\ref{AsLog}.
According to the symmetry condition \eqref{a2}  if $\Im\zeta_{\ell} \neq 0$, then together with $\omega_{\ell}(\mu/\zeta_{\ell})$, the symbol necessarily contains the term $  \overline{\omega_{\ell}(\overline\mu/\zeta_{\ell})}$.

 In view of Theorem~\ref{lma.be}, the result below follows directly from Theorem~\ref{thm.a5}.

\begin{theorem}\label{HT}
Let  functions $ \phi_{1} ,  \phi_{-1} , \omega_1,\dots, \omega_L$ satisfy Assumption~$\ref{AsLog}$.  Suppose that
$$
\omega(\mu)
= 
 \phi_1(\mu )+  \phi_{-1}(-\mu )+ \sum_{\ell=1}^L\bigl(\omega_{\ell}(\mu/\zeta_{\ell}) + \overline{\omega_{\ell}(\overline\mu/\zeta_{\ell})}\bigr) +\wt \omega (\mu)
$$
 where $  \zeta_1,\dots,\zeta_L\in\bbT$ are distinct numbers with $\Im \zeta_{\ell}>0$ and the remainder $\wt \omega \in L^2  $ and $P_{+}\wt \omega $ belongs to the Besov space $ B^\alpha_{1/\alpha, 1/\alpha} (\bbT)$. 
We assume that the functions $\phi_1$,  $\phi_{-1}$ and $\wt \omega$ satisfy the symmetry condition \eqref{a2}.
Let  the numbers $b_{1}, \ldots, b_L$  be the asymptotic coefficients for the functions $\omega_1,\dots,\omega_L$, defined by \eqref{b4ab}, and let ${\sf b}_{1} $, ${\sf b}_{-1}$ be the coefficients for 
$\phi_1$, $\phi_{-1}$, defined by  \eqref{b4abS}. 
Finally, let the coefficient $a^\pm$ be given by \eqref{a19}.
Then  the Hankel operator
$H (\omega )$ is compact and its    eigenvalues   have the asymptotic behavior
$$
\lambda_n^\pm (H (\omega))= a^\pm \, n^{-\alpha}+o(n^{-\alpha})
$$
as $n\to\infty$.   
\end{theorem}

We refer to the book \cite{Peller}, Appendix~2,  for the precise definition of Besov classes. Note also that the conditions on the remainder  
$\wt{\omega}$ can be stated (see \cite{IV}) in a more explicit although less sharp form. For example, it suffices to suppose that 
$$
\wt\omega(\mu)
= 
 \wt\phi_1(\mu )+  \wt\phi_{-1}(-\mu )+ \sum_{\ell=1}^L\bigl(\wt\omega_{\ell}(\mu/\zeta_{\ell}) + \overline{\wt\omega_{\ell}(\overline\mu/\zeta_{\ell})}\bigr)  
$$
where $ \wt\phi_{1} ,  \wt\phi_{-1} , \wt\omega_1,\dots, \wt\omega_L$ satisfy Assumption~$\ref{AsLog}$ for some $\beta>\alpha$.

Observe that the function $\omega$ in \eqref{b4cx} is unbounded if $\alpha<1$. Nevertheless according to Theorem~\ref{HT} the corresponding operator $H(\omega)$ is compact. This is of course consistent with the Hartman theorem (see \cite{Peller}, Chapter~1.5) which guarantees that $H(\omega)$ is compact if $\omega\in \VMO (\bbT)$ (the class of functions of vanishing mean oscillation).

\section{Spectral asymptotics for Hankel operators in $L^2(\bbR_+)$}\label{sec.e}

The main result of this section is stated as Theorem~\ref{thm.a8} where kernels $\bh (t)$ are singular both for $t\to \infty$ and for $t\to 0$. We also consider   (see Theorem~\ref{sing}) kernels   with singularities at   points $t_{0}>0$  (instead of $t_{0}=0$).

\subsection{Basic definitions}

Integral Hankel operators $\bGank(\bh)$ in the space  $L^2(\bbR_+)$
  are
formally defined by the relation
$$
(\bGank(\bh)\bu)(t)=\int_0^\infty \bh(t+s)\bu (s)ds, \quad {\bf u}\in C_{0}^\infty (\bbR_{+}),
$$
where $\bh\in L^1_\loc(\bbR_+)$; this function is called the \emph{kernel} of the Hankel operator $\bGank(\bh)$. 
Under the assumptions below the operator $\bGank(\bh)$ are compact.
Of course the operator $\bGank(\bh)$  is self-adjoint 
if and only if  the function $\bh(t)$ is real valued.

Similarly to the discrete case, bounded Hankel operators  $\bGank(\bh)$ are  unitarily equivalent to the operators $\bHank( {\bomega} )$ defined by formula \eqref{b6}
in the Hardy space $H^2(\bbR)$:
\begin{equation}
\Phi \bHank( {\bomega} )\Phi^* = \bGank(\bh)
\quad \text{ if }\quad
\bh(t)=\frac1{\sqrt{2\pi}}  \wh{\bomega} (t)
\quad \text{for $t>0$,}
\label{e4}
\end{equation}
where $\Phi$ is the Fourier transform \eqref{b.ft}. 
The Fourier transform $\wh{\bomega}$ of $ \bomega\in L^\infty (\bbR)$ should in general be understood 
in the sense of distributions (for example, on the Schwartz class ${\mathcal S}' (\bbR)$) and 
the precise meaning of \eqref{e4} is given by the equation
$$
(\bHank({\bomega})\Phi^* {\bf u} , \Phi^* {\bf u}) 
=
( \bGank(\bh) {\bf u},{\bf u}), \quad {\bf u}\in C_{0}^\infty (\bbR_{+}).
$$
A function $\bomega (x)$ satisfying  the second equality \eqref{e4} is known as a symbol 
of the Hankel operator $\bGank(\bh)$.

In the discrete case, the spectral asymptotics of $\Gank(h)$ is determined by the 
behavior of the sequence $h(j)$ as $j\to\infty$. 
In the continuous case, the behavior of the kernel $\bh(t)$ for $t\to \infty$ and for $t\to 0$
as well as the singularities of $\bh (t)$  at points $t_{0}>0$
contribute to the spectral properties of $\bGank(\bh)$.

\subsection{Previous results}

We fix two functions $\chi_0,\chi_\infty\in C^\infty(\bbR_+)$ such that
$$
\chi_0(x)=
\begin{cases}
1& \text{for $\abs{x}\leq c_1$,}
\\
0& \text{for  $\abs{x}\geq c_2$,}
\end{cases}
\quad
\chi_\infty(x)=
\begin{cases}
0& \text{for $\abs{x}\leq C_1$,}
\\
1& \text{for  $\abs{x}\geq C_2$,}
\end{cases}
$$
for some $0<c_1<c_2<1$ and $1<C_1<C_2$, 
and define the model kernels 
\begin{equation}
\bq_0(t)=\chi_0(t)t^{-1}(\log(1/t))^{-\alpha}, 
\quad
\bq_\infty(t)=\chi_\infty(t)t^{-1}(\log t)^{-\alpha}, 
\quad 
t>0.
\label{a26}
\end{equation}
As usual, the coefficient $\varkappa(\alpha)$ is given by \eqref{a15}.

\begin{theorem}\label{thm.a9}\cite[Theorem~3.1]{II}
Let $\alpha>0$.
Then  
$$
\lambda_n^+(\bGank(\bq_0))= 
\varkappa(\alpha)n^{-\alpha} + o(n^{-\alpha})
 \quad \text{and}
\quad
\lambda_n^+(\bGank(\bq_\infty))= 
\varkappa(\alpha)n^{-\alpha} + o(n^{-\alpha})
$$
as $n\to\infty$. Moreover,
 $$
 \lambda_n^-(\bGank(\bq_0))= 
  o(n^{-\alpha})
 \quad \text{and}
\quad
\lambda_n^-(\bGank(\bq_\infty))= 
  o(n^{-\alpha}).
$$
\end{theorem}

Let us construct symbols corresponding to the kernels \eqref{a26}.
\begin{lemma}\label{lma.e1}\cite[Lemma~6.3]{III}
Let $\bomega_0$ and $\bomega_\infty$ be defined by 
\begin{equation}
\bomega_0(x)=2i\int_0^\infty \bq_0(t)\sin(xt)dt, 
\quad
\bomega_\infty(x)=2i\int_0^\infty \bq_\infty(t)\sin(xt)dt, \quad x\in \bbR,
\label{a27z}
\end{equation}
where $\bq_0(t)$ and $\bq_\infty(t)$ are  given by \eqref{a26} with $\alpha\geq 0$. Then 
$\bomega_0,\bomega_\infty \in L^\infty (\bbR)$ and $\bomega_0 \in C^\infty (\bbR)$,
$ \bomega_\infty \in C^\infty (\bbR_*\setminus\{0\})$. For $t> 0$, we have
$$
\bq_{0}(t)=\frac1{\sqrt{2\pi}}  \wh{\bomega}_{0} (t)\quad \text{and}\quad
\bq_{\infty}(t)=\frac1{\sqrt{2\pi}}  \wh{\bomega}_{\infty} (t).
$$
\end{lemma}

The assertion below is a particular case of our general result (Theorem~5.1 in \cite{III}) 
on the asymptotics of singular values of integral Hankel operators, needed in the present text.

\begin{theorem}\label{thm.a10}
Let $\alpha>0$, let $\rho \in\bbR$, $\rho\not=0$, and let   
$\bb \in\bbC$ be arbitrary. 
If
$$
\bh(t)
= 2 \Re
(  \bb  e^{-i \rho  t})\bq_\infty(t) ,
$$
then   
\begin{equation}
s_n(\bGank(\bh))= 
2^\alpha \varkappa(\alpha) | \bb |n^{-\alpha} + o(n^{-\alpha}).
\label{a28E}
\end{equation}
\end{theorem}

Using the symmetry principle, we get   the following result.

\begin{theorem}\label{thm.a10E}
Let the   function $\bh(t)$ be the same as in Theorem~$\ref{thm.a10}$.
Then  
\begin{equation}
\lambda^\pm_n(\bGank(\bh))= 
\varkappa(\alpha)  | \bb | n^{-\alpha} + o(n^{-\alpha}).
\label{a28}
\end{equation}
\end{theorem}

\begin{proof} 
Let the symbol $\bomega_\infty$ be defined by \eqref{a27z} and let
$$
\bomega(x)=\bb\, \bomega_\infty(x-\rho)+\overline{\bb}\,\bomega_\infty(x+\rho).
$$
Then $\wh\bomega(t)=\sqrt{2\pi}\bh(t)$ for $t>0$, and hence the operators
  $\bGank(\bh)$ and $\bHank(\bomega)$ are unitarily equivalent.
By Lemma~\ref{lma.e1},  the singular support of the symbol $\bomega$ consists of the pair of points $\rho,-\rho$.  
Therefore by the symmetry principle
 (Theorem~\ref{thm.c3}) we have 
$$
\Delta_p^+(\bGank(\bh))=\Delta_p^-(\bGank(\bh))=\tfrac12\Delta_p(\bGank(\bh)),
$$
and similarly for the lower limits. The asymptotic relation \eqref{a28E} for the singular values 
can be equivalently rewritten as 
$\Delta_p(\bGank(\bh))=2 \varkappa(\alpha)^p  | \bb |^p $, and thus we obtain
$$
\Delta_p^+(\bGank(\bh))=\Delta_p^-(\bGank(\bh))=\varkappa(\alpha)^p  | \bb |^p
$$
and similarly for the lower limits. This yields \eqref{a28}.
\end{proof}

 Note that Theorems~\ref{thm.a10} and \ref{thm.a10E} are the analogues of Theorems~\ref{thm.a4} and \ref{CC} in the continuous case. The following  result concerning the error term is the analogue of Theorem~\ref{thm.a5a}. 
Below $\jap{x}=\sqrt{1+\abs{x}^2}$.

\begin{theorem}\label{thm.a8}
Let $\alpha>0$ and  let $M=M(\alpha)$ be the integer given by \eqref{a16}.
Let $\bg $  be a complex valued function in $L^\infty_\loc(\bbR_+)$; 
if $\alpha\geq 1/2$, suppose also 
that $\bg \in C^M(\bbR_+)$. 
Assume that    for all $m=0,\dots,M$, we have
\begin{equation}
 \bg^{(m)} (t)=o( t^{-1-m}\jap{\log t}^{-\alpha}) 
\label{a23}
\end{equation}
 as $t\to \infty$ and as $t\to 0$.
Pick
 any   $\rho \in\bbR$ and put $\mathbf{g}_\rho (t)= e^{-i\rho  t}\bg (t)$. Then $s_{n} (\bGank(\mathbf{g}_\rho)) =o(n^{-\alpha})$. 
\end{theorem}


\subsection{Asymptotics of eigenvalues} 
Our main result   concerns real kernels $\bh(t)$ that are singular at $t=0$ and contain several oscillating terms at  infinity. The assertion below is the analogue of Theorem~\ref{thm.a5}, and its proof follows the same steps.

\begin{theorem}\label{thm.a7}
Let $\alpha>0$, let $\rho_1,\dots,\rho_L $ be distinct positive numbers,  
and let $\bb_0,\bb_1,\dots,\bb_L,\bb_\infty$ as well as
$\psi_1,\dots,\psi_L$ be any real numbers. 
Let the number $M=M(\alpha)$ be   given by \eqref{a16}.
Suppose that $\bh\in L^\infty_{\text{loc}} (\bbR_{+})$ if $\alpha<1/2$ and $\bh\in C^M (\bbR_{+})$  if $\alpha\geq 1/2$.
Assume  that
\begin{align}
\bh(t)= \bb_\infty  t^{-1}(\log t)^{-\alpha}+&\bg_\infty(t)
+2\sum_{\ell=1}^L \bigl(\bb_\ell  t^{-1}(\log t)^{-\alpha}+\bg_\ell(t)\bigr)\cos(\rho_\ell t-\psi_\ell)
\quad t\geq2,
\label{z2}
\\
\bh(t)&=\bb_0t^{-1}\bigl(\log(1/t)\bigr)^{-\alpha}+  \bg_0(t), 
\quad t\leq 1/2,
\label{z3}
\end{align}
where the error terms $\bg_\infty$, $\bg_1, \ldots, \bg_L$ obey the estimates
\eqref{a23}
as $ t\to\infty$  and $\bg_0$ obeys these estimates
 as $t\to 0$. 
Then the eigenvalues of the integral Hankel operator $\bGank(\bh)$ 
satisfy the asymptotic relation 
\begin{equation}
\lambda_{n}^\pm(\bGank(\bh)) 
=
{\bf a}^\pm n^{-\alpha} + o(n^{-\alpha})
\label{a22}
\end{equation}
where 
\begin{equation}
{\bf a}^\pm 
=
\varkappa(\alpha)\bigl((\bb_0)_\pm^{1/\alpha}+(\bb_\infty )_\pm^{1/\alpha} +\sum_{\ell= 1}^L \abs{\bb_\ell}^{1/\alpha}\bigr)^\alpha
\label{a22E}
\end{equation}
and the coefficient $\varkappa(\alpha)$ is given by \eqref{a15}.
\end{theorem}

\begin{proof}
 We first consider every term in the right-hand sides of \eqref{z2} and \eqref{z3} separately. Recall that the functions
 $\bq_0(t)$ and $\bq_{\infty}(t)$ are defined by formulas \eqref{a26}.  Put
$$
\bh_{0}(t)= \bb_0  \bq_{0}(t), 
\quad
\bh_{\infty}(t)= \bb_\infty  \bq_{\infty}(t), 
$$
and 
$$
\bh_\ell(t)
=
2\bb_\ell\cos(\rho_\ell t-\psi_\ell)\bq_{\infty}(t)
=
2\Re \bigl(\bb_\ell  e^{i\psi_\ell}e^{-i \rho_\ell t}\bigr)  \bq_{\infty}(t), \quad\ell=1,\ldots, L. 
$$
Similarly to Theorem~\ref{thm.a5}, we make the reasoning in terms of the functionals $\Delta_p^\pm$, 
$\delta_p^\pm$, where $p=1/\alpha$.
By Theorem~\ref{thm.a9}, we have
\begin{equation}
\begin{split}
\Delta_p^\pm(\bGank(\bh_{0}))
=
\delta_p^\pm(\bGank(\bh_{0}))
&=
(\varkappa(\alpha)\bb_{0})_\pm^p,
\\
\Delta_p^\pm(\bGank(\bh_\infty))
=
\delta_p^\pm(\bGank(\bh_\infty))
&=
(\varkappa(\alpha)\bb_\infty)_\pm^p,
\label{d1R}
\end{split}
\end{equation}
and by  Theorem~\ref{thm.a10E}, we have
\begin{equation}
\Delta_p^\pm(\bGank(\bh_\ell))= 
\delta_p^\pm(\bGank(\bh_\ell))= 
 \abs{\varkappa(\alpha)\bb_\ell}^p, \quad \ell=1,\ldots, L.
\label{d3R}
\end{equation}

It follows from Lemma~\ref{lma.e1} that the singular supports of  the symbols of the operators $\bGank(\bh_{0})$ and  $\bGank(\bh_{\infty}) $ consist of the points $\infty$ and $0$, respectively. 
Lemma~\ref{lma.e1}  also implies that
the singular supports of the symbols of the operators $\bGank(\bh_{\ell})$   consist of the pairs  $\{-\rho_\ell, \rho_\ell\}$.
So we  can apply Theorem~\ref{thm.b6} (the localization principle for eigenvalues) to the Hankel operator $\bGank(\bh_*)$ with
\begin{equation}
\bh_*=\bh_0+ \bh_{\infty}+\sum_{\ell=1}^L \bh_\ell
\label{eq:h1}
\end{equation}
which yields
$$
\Delta_p^\pm(\bGank(\bh_*))
=
\delta_p^\pm(\bGank(\bh_*))
=
\Delta_p^\pm(\bGank(\bh_{0}))
+
\Delta_p^\pm(\bGank(\bh_\infty))
+
\sum_{\ell=1}^L \Delta_p^\pm(\bGank(\bh_\ell)).
$$
Now relations \eqref{d1R}  and \eqref{d3R} imply that
\begin{equation}
\Delta_p^\pm(\bGank(\bh_*))
= \delta_p^\pm(\bGank(\bh_*))
=
\varkappa(\alpha)^p\bigl((\bb_{0})_\pm^p+(\bb_{\infty})_\pm^p+\sum_{\ell=1}^L \abs{\bb_\ell}^p\bigr).
\label{ZZ}\end{equation}

Finally, put $\bg=\bh -\bh_{*}$. Using representations \eqref{z2}, \eqref{z3}, our conditions on  $\bg_\infty$, $\bg_1, \ldots, \bg_L$,  $\bg_0$  and applying Theorem~\ref{thm.a8}, we see that $ \bGank(\bg)\in\Sch_{p,\infty}^0$. Since
$\bGank(\bh)=\bGank(\bh_*)+\bGank(\bg)$, it follows from   Lemma~\ref{lma.b1}  that 
$$
\Delta_p^\pm(\bGank(\bh))=\Delta_p^\pm( \bGank(\bh_*))\quad \text{and} \quad \delta_p^\pm(\bGank(\bh))=\delta_p^\pm(\bGank(\bh_*)).
$$ 
Now using
\eqref{ZZ}, we obtain the relations  \eqref{a22},  \eqref{a22E}. 
\end{proof}

\subsection{Local singularities of the kernel}

The localization principle shows that the results on the asymptotics of eigenvalues of different Hankel operators can be combined provided that the singular supports  of their symbols are disjoint.  
This idea has already been illustrated by Theorems~\ref{thm.a4} and \ref{thm.a7}. 
Here we apply the same arguments to  kernels ${\sh }(t)$ satisfying condition \eqref{z2} as $t\to\infty$ and    singular at some  positive point.

The effect of local singularities of a kernel on the asymptotics 
of eigenvalues of the corresponding Hankel operator  
was studied  in \cite{Yafaev2}.  

\begin{lemma}\label{sing1}\cite[Lemma~6.2]{Yafaev2}
Let $t_{0}>0$, $m\in {\bbZ}_{+}$  and  
\begin{equation}
{\bff}  (t)=   (t_0-t)^{m} \quad \text{for} \quad t\leq t_{0}, \quad {\bff}  (t)=  0 \quad \text{for} \quad t >t_{0}. 
\label{eq:sing1}\end{equation}
Then  
\begin{equation}
\lambda_{n}^\pm  (\bGank({\bff}  ))
= m!   t_{0}^{m+1} (2\pi n)^{-m-1} (1+O(n^{-1})),
\quad
n\to \infty.
\label{eq:sing7}\end{equation}
\end{lemma}

We also note the explicit formula for the symbol ${\boldsymbol{\tau}}_{m}(x)$  of the operator $\bGank({\bff} )$:
 \begin{equation}
 {\boldsymbol{\tau}}_{m}(x)= m! (ix)^{-m-1} \bigl( e^{i t_0 x } -\sum_{k=0}^m \frac{1} {k!} (i t_{0} x)^k
\bigr), \quad x\in\bbR.
\label{eq:sing6}
\end{equation}
Obviously, $ {\boldsymbol{\tau}}_{m}\in C^\infty (\bbR)$ and $ {\boldsymbol{\tau}}_{m}(x)$ is an oscillating function as $|x|\to\infty$. Therefore $\sing\supp \boldsymbol{\tau}_{m} = \{ \infty \}$, and hence the symmetry principle
(Theorem~\ref{thm.c3}) cannot be applied to the operator $\bGank({\bff})$. Nevertheless according to \eqref{eq:sing7}  its spectrum is asymptotically symmetric.

We are now in a position to consider the general case. 

\begin{theorem}\label{sing}
Let $t_{0}>0$, $m\in {\bbZ}_{+}$ and $\tt b\in {\bbR}$.  Set
$$
{\sh} (t)= {\tt b}  {\bff}  (t)+  \bh (t)
$$
where ${\bff}  (t)$ is given by   \eqref{eq:sing1} and $\bh (t)$ satisfies the assumptions of Theorem~$\ref{thm.a7} $ with $\bb_{0}=0$ and $\alpha=m+1$. 
Then the eigenvalues of the integral Hankel operator $\bGank({{\sh} })$ 
satisfy the asymptotic relation 
\begin{equation}
\lambda^\pm_{n} (\bGank({{\sh} }))= {\bf{\sf a}}^\pm n^{-m-1} + o(n^{-m-1})
\label{eq:Sloc}
\end{equation}
with
$$
{\bf{\sf a}}^\pm= \Bigl( (2\pi)^{-1}  t_0 (m! |{\tt b}  |)^{1/\alpha }+\varkappa (\alpha)^{1/\alpha}\bigl((\bb_{\infty})_\pm^{1/\alpha}
+ \sum_{\ell=1}^L \abs{\bb_\ell}^{1/\alpha}\bigr)\Bigr)^\alpha,
\quad \alpha=m+1.
$$
\end{theorem}

The proof of this result is practically the same as that of
Theorem~\ref{thm.a7}. The only difference is that the term 
$\bh_{0} (t)$ should be replaced by  $ {\tt b}  {\bff}  (t)  $ in \eqref{eq:h1}.

 Observe that we have excluded the term  \eqref{z3} singular at $t=0$ in Theorem~\ref{sing}  because the corresponding symbol is singular
 at the same point $x=\infty$ as the function \eqref{eq:sing6}. 
 In this case one might expect that the   contributions of singularities of ${\sh}(t)$ at $t=0$ and $t=t_{0} >0$ are not independent of each other.
 In any case, our technique does not allow us to treat this situation.  Finally, we note that
we have chosen $\alpha=m+1$ in Theorem~\ref{sing} since in this case both the local singularity of ${\sh}(t)$ at $t=t_{0}$ and its ``tail"   as $t\to\infty$ contribute to the asymptotic coefficient ${\bf{\sf a} }^\pm $ in \eqref{eq:Sloc}.

Similarly to Section~5, Theorems~\ref{thm.a7} and \ref{sing} can also be reformulated in terms of Hankel operators in the Hardy space $H^2(\bbR)$, but we do not dwell upon it here.

\end{document}